\documentclass{conm-p-l}

\usepackage{amsthm,epsf,amsfonts}
\input{amssymb.sty}
\usepackage{amscd,amsmath}
\usepackage[arrow, matrix]{xy}
\usepackage[pdftex]{graphicx}

\makeatletter

\theoremstyle{plain}
\newtheorem{thm}{Theorem}[section]
\numberwithin{equation}{section} 
\numberwithin{figure}{section} 
\theoremstyle{plain}
\newtheorem*{thm*}{Theorem}
\theoremstyle{plain}
\theoremstyle{plain}
\newtheorem*{cor*}{Corollary}

\theoremstyle{plain}
\theoremstyle{plain}
\theoremstyle{definition}

\theoremstyle{remark}

\theoremstyle{remark}

\theoremstyle{remark}

\theoremstyle{remark}

\theoremstyle{definition}

\theoremstyle{remark}
\newtheorem*{acknowledgement*}{Acknowledgement}

\theoremstyle{plain}
\newtheorem{subthm}{Theorem}[subsection]
\theoremstyle{plain}
\newtheorem{subcor}[subthm]{Corollary} 
\theoremstyle{plain}
\newtheorem{sublem}[subthm]{Lemma} 
\theoremstyle{plain}
\newtheorem{subprop}[subthm]{Proposition} 
\theoremstyle{definition}
\newtheorem{subdefn}[subthm]{Definition}
\theoremstyle{remark}
\newtheorem{subrem}[subthm]{Remark}
\theoremstyle{remark}

\theoremstyle{remark}

\theoremstyle{plain}
\newtheorem{subexer}[subthm]{Exercise}
\newtheorem{subexample}[subthm]{Example}

\newcommand{\R}{{\mathbb R}}
\newcommand{\C}{{\mathbb C}}
\newcommand{\N}{{\mathbb N}}
\newcommand{\Z}{{\mathbb Z}}

\newcommand{\M}{{\mathbb M}}
\newcommand{\K}{{\mathbb K}}
\newcommand{\B}{{\mathbb B}}
\newcommand{\F}{{\mathbb F}}
\newcommand{\G}{\Gamma}
\newcommand{\e}{\varepsilon}
\newcommand{\p}{\varphi}
\newcommand{\hh}{{\mathcal H}}
\newcommand{\hk}{{\mathcal K}}
\newcommand{\id}{\mathrm{id}}

\DeclareMathOperator{\Tr}{Tr}

\newcommand{\mtp}{\mathbin{\otimes}_{\max}}
\newcommand{\vt}{\mathbin{\bar{\otimes}}}
\newcommand{\fin}{{\mathfrak{F}}}

\DeclareMathOperator{\Aut}{Aut}

\DeclareMathOperator{\prob}{Prob}

\newcommand{\grp}{\mathcal{G}}

\def\freeprod{\font\bigsymbolsfont=cmsy10 scaled \magstep3
 \setbox0=\hbox{\bigsymbolsfont\char'003 }\mathord{\lower1pt\box0}}\relax\ignorespaces

\usepackage{amssymb} \usepackage{amscd}

\makeatother

\begin{document}

\title[The symbiosis of C$^*$- and W$^*$-algebras]{The symbiosis of C$^*$- and W$^*$-algebras}

\author{Nathanial P. Brown}

\address{Department of Mathematics, Penn State University, State
College, PA 16802}

\email{nbrown@math.psu.edu}


\thanks{Partially supported by DMS-0554870.}

\begin{abstract}  These days it is common for young operator algebraists to know a lot about C$^*$-algebras, or a lot about von Neumann algebras -- but not both.  Though a natural consequence of the breadth and depth of each subject, this is unfortunate as the interplay between the two theories has deep historical roots and has led to many beautiful results.  We review some of these connections, in the context of amenability, with the hope of convincing (younger) readers that tribalism impedes progress.  
\end{abstract}

\maketitle

\section{Introduction} 

I was raised a hardcore C$^*$-algebraist.  My thesis focused on C$^*$-dynamical systems, and never once required a weak topology.  As a fresh PhD my knowledge of von Neumann algebras was superficial, at best.    I didn't really like von Neumann algebras, didn't understand them, and certainly didn't need them to prove theorems.  Conversations with new W$^*$-PhDs make it clear that this goes both ways; they often know little about C$^*$-algebras, and care less.  

Today I still know relatively little about von Neumann algebras, but I have grown to love them.  And they are an indispensable tool in my C$^*$-research.  Conversely, recent work of Narutaka Ozawa (some in collaboration with Sorin Popa) has shown that C$^*$-techniques can have deep applications to the structure theory of certain von Neumann algebras.  In other words, there are very good reasons for C$^*$-algebraists and von Neumann algebraists to learn something about each other's craft.   In the ``old" days (say 30 or more years ago), the previous sentence would have been silly (indeed, some ``old" timers may still find it silly) as the field of operator algebras was small enough for students to become well acquainted with most of it.  That's no longer the case.  Hence, I hope these notes will help my generation, and those that follow, to see the delightfully intertwined theories of C$^*$- and W$^*$-algebras as an indivisible unit. 

I do not intend to write an encyclopedia of C$^*$- and W$^*$-interactions.  Amenability (for groups, actions and operator algebras) is a perfect context for illustrating some of the most important interactions, so these notes are organized around that theme.  Also, I assume familiarity with basics such as C$^*$- and W$^*$-algebras associated to discrete groups, crossed products, and numerous other things.  Keeping these notes self contained would, I fear, bury the ideas being advertised  under a mountain of details.  Hence, the reader will frequently be referred to \cite{Me-n-Taka} (another advertisement, of the shameless sort) for more details.\footnote{In fact, large portions of these notes are shamelessly cut-n-pasted directly from \cite{Me-n-Taka}! And I gratefully acknowledge the AMS's permission to do so.}    

\vspace{2mm}

{\noindent\textbf{Acknowlegment:}} These notes are based on lectures given at the Summer School on Aspects of Operator Algebras and Applications (Palacio de la Magdalena, Santander) and the workshop Operator Algebras, Dynamics and Classification (Texas A\&M).  The organizers at both places were amazing and I can't thank them enough for their hospitality.

\subsection{C$^*$-algebras vs.\ W$^*$-algebras}   

I distinctly remember my first encounter with a W$^*$-Fundamentalist.  After a day at GPOTS, we were at a bar impressing locals with exotic vocabulary. Out of nowhere came the verbal left hook.  

``Why do you study C$^*$-algebras?" he asked with palpable contempt. ``They're not interesting, you can't do anything with them...they don't even have a Borel functional calculus!"  

I was stunned silent. (No small accomplishment, given my big mouth.) But later that night, lying awake and trembling with rage, I finally had a piercing come back. 

``Anybody can prove a theorem about von Neumann algebras. How hard is that? Look at all the tools you give yourself: Weakly compact unit balls, polar decompositions, you even have a Borel functional calculus!   A child could do something with all those tools. And another thing: You're a jerk!"  

Pleased at having silenced the voice in my head, I fell asleep with a smug smile. ;-) 

\vspace{5mm}

Fundamentalism is stupid.  Trust me, I know, I'm a recovering Fundamentalist.  It took me a long time to see that I was guilty of precisely the sort of tribalism that I now wish to discourage.  (And, yes, I occasionally relapse.) These days when I meet a W$^*$-Fundamentalist, I try to point out what C$^*$-algebras have done for von Neumann algebras.   For example, having complicated representation theory is not a C$^*$-defect -- it's an opportunity.  Just ask Bob Powers.  His celebrated construction of non-isomorphic type III factors (cf.\ \cite{powers}) depends on the non-triviality of C$^*$-representation theory.  In the latter half of these notes we'll see more recent examples; Ozawa's work on solid von Neumann algebras is largely C$^*$-algebraic. 

For my own part, I no longer denigrate the rich structure of von Neumann algebras. I exploit it.  I love the fact that they have projections and (in the finite case) tracial states. I'm green with envy over the compactness of their unit balls.  For example, how awesome would it be if the following result, which is a fairly simple consequence of Alaoglu's Theorem, had a C$^*$-analogue? 

\begin{subthm}
\label{subthm:pointultraweak} Let $X$ be a  Banach space, $M$ be a von Neumann algebra and $T_{\lambda}\colon X \to M$ be a bounded net of linear maps. Then $\{ T_{\lambda}\}_{\lambda \in \Lambda}$ has a cluster point in the point-ultraweak topology.
\end{subthm}

Don't see the point?  Well, imagine the theorems you could prove if, for example, every sequence of asymptotically multiplicative u.c.p.\ maps $\p_n \colon A \to B$ had a point-norm convergent subsequence -- i.e., gave rise to a $*$-homomorphism $A \to B$! (For example, you could prove that all C$^*$-algebras are type I, a fantastically false result.) 

The point I'm trying to make is that in a von Neumann algebra you can do things we C$^*$-Fundamentalists only dream of doing. Passing to a weak closure and gettin' crazy can be a lot of fun -- not to mention  fruitful.  Indeed, when trying to prove a C$^*$-theorem my first thought is usually this:  ``How do I translate this problem into a W$^*$-question? And how will I come back?"  Of course, certain C$^*$-questions are more amenable to this approach than others, but in these notes I'll try to explain some of the ways that von Neumann algebras have been used to prove C$^*$-theorems, and vice versa.  The results are impressive, by any standard.  And beautiful!  

I hope you enjoy them as much as I do.

\subsection{Five classical theorems}  

Here are some general tools that facilitate the passage between norm-closed and  weakly-closed algebras.  The first result, one of the oldest in the subject, is still used daily. 

\begin{subthm}[Bicommutant Theorem] Let $A \subset \B(\hh)$ be a C$^*$-algebra acting nondegenerately.  The weak-operator-topology closure of $A$ is equal to the double commutant $A^{\prime\prime}$. 
\end{subthm}

\begin{subthm}[Kaplansky's Density Theorem] Let $A \subset \B(\hh)$ be a C$^*$-algebra acting nondegenerately.  Then the unit ball of $A$ is weakly dense in the unit ball of $A^{\prime\prime}$. 
\end{subthm} 

\begin{subthm}[Up-Down Theorem]  Let $A \subset \B(\hh)$ be a C$^*$-algebra acting nondegenerately on a separable Hilbert space $\hh$.  For each self-adjoint $x \in A^{\prime\prime}$, there exists a decreasing sequence of self-adjoints $x_n \geq x_{n+1} \geq \cdots$ in $A^{\prime\prime}$ such that
\begin{enumerate} 
\item $x_n \to x$ in the strong operator topology, and  

\item for each $n \in \N$, there exists an increasing sequence of self-adjoint $y^{(n)}_k \leq y^{(n)}_{k+1}$ in $A$, such that $y^{(n)}_k \to x_n$ (as $k \to \infty$) in the strong operator topology. 
\end{enumerate} 
\end{subthm} 

\begin{subthm}[Lusin's Theorem]   Let $A \subset \B(\hh)$ be a C$^*$-algebra acting nondegenerately. For every
finite set of vectors $\fin \subset \hh$, $\e > 0$, projection $p_0
\in A^{\prime\prime}$ and self-adjoint $y \in A^{\prime\prime}$, there exist a self-adjoint $x \in
A$ and a projection $p \in A^{\prime\prime}$ such that $p \leq p_0$, $\|p(h) -
p_0(h)\| < \e$ for all $h \in \fin$, $\|x\| \leq \min\{2\|yp_0\|,\|y\|\} + \e$ and
$xp = yp$.
\end{subthm} 

\begin{subthm}[Double Dual Theorem] The (Banach space) double dual $A^{**}$ of a C$^*$-algebra $A$ is a von Neumann algebra. Moreover, the ultraweak topology on $A^{**}$ (coming from its von Neumann algebra structure) agrees with the weak-$*$ topology (coming from $A^*$), and hence restricts to the weak topology on $A$ (coming from $A^*$).
\end{subthm} 

From a C$^*$-algebraist's point of view, the double dual theorem is probably the most important as it allows one to come back from the world of von Neumann algebras.  That is, suppose one wants to prove a C$^*$-theorem, exploiting the enormous W$^*$-toolkit. Well, in any representation of the given C$^*$-algebra one could take the weak closure and bang away.  But how to come back to the C$^*$-algebra of interest?   

Answer: The Hahn-Banach Theorem.  

That is, rather than work in any old weak closure, one should work in the double dual von Neumann algebra, where the Hahn-Banach theorem implies that convex sets have the same norm and weak closures.  

That probably makes little sense, so let's look at an illustrative example. Let's show that if $A^{**}$ is semidiscrete, then $A$ is nuclear.

\subsection{Semidiscreteness and Nuclearity}  
Recall that a linear map $\p\colon A \to B$ is \emph{completely positive} (c.p.) if the matrix-level map $\p_n\colon \M_n(A) \to \M_n(B)$, defined by
\[
\p_n([a_{i,j}]) = [\p(a_{i,j})],
\]
is positive (i.e., maps positive matrices to positive matrices) for every $n \in \N$.   Special attention will be paid to the cases where $\p$ is unital (u.c.p.) or, more generally, contractive (c.c.p.). 



\begin{subdefn}  
\label{defn:nuclear}
A map $\theta\colon A \to B$ is \emph{nuclear} if there exist c.c.p.\ maps $\p_n \colon A \to \M_{k(n)} (\C)$ and $\psi_n\colon \M_{k(n)} (\C) \to B$ such that $\psi_n \circ \p_n(a) \to \theta(a)$ in norm, for all $a \in A$.\footnote{One could also use general finite-dimensional algebras in this definition -- see \cite[Exercise 2.1.2]{Me-n-Taka}.}


A C$^*$-algebra $A$ is \emph{nuclear} if the identity map $\mathrm{id}\colon A \to A$ is nuclear, i.e., if there exist c.c.p.\ maps $\p_n \colon A \to \M_{k(n)} (\C)$ and $\psi_n\colon \M_{k(n)} (\C) \to A$ such that $\psi_n \circ \p_n(a) \to a$ in norm, for all $a \in A$. 

A W$^*$-algebra $M$ is \emph{semidiscrete} if there exist c.c.p.\ maps $\p_n \colon M \to \M_{k(n)} (\C)$ and $\psi_n\colon \M_{k(n)} (\C) \to M$ such that $\psi_n \circ \p_n(x) \to x$ ultraweakly, for all $x \in M$. 
\end{subdefn} 

To show ``$A^{**}$ semidiscrete $\Longrightarrow A$ nuclear," we need an important matrix-fact. 

\begin{subprop}\label{prop:cpmapsfrommatrix}
Let $A$ be a $\mathrm{C}^*$-algebra and $\{e_{i,j}\}$ be 
matrix units of $\M_n(\C)$. A map $\p\colon \M_n(\C)\to A$ is
c.p.\ if and only if $[\p(e_{i,j})]_{i,j}$ is positive in $\M_n(A)$. 
\end{subprop}

\begin{subprop} 
\label{prop:seminuc}
If $A^{**}$ is semidiscrete, then $A$ is nuclear. 
\end{subprop} 

\begin{proof} I'll sketch the argument, highlighting the use of the double dual theorem and neglecting some nontrivial details (see \cite[Proposition 2.3.8]{Me-n-Taka}).  

Let $\p_n \colon A^{**} \to \M_{k(n)} (\C)$, $\psi_n\colon \M_{k(n)} (\C) \to A^{**}$ be such that $\psi_n \circ \p_n(x) \to x$ ultraweakly, for all $x \in A^{**}$.  An approximation argument, using Proposition \ref{prop:cpmapsfrommatrix} and the fact that $\M_k(A)$ is ultraweakly dense in $\M_k(A^{**})$ for all $k \in \N$, allows us to assume that $\psi_n(\M_{k(n)} (\C)) \subset A$ for all $n$.  

Here's the punchline: For each $a \in A$, since the ultraweak topology on $A^{**}$ restricts to the weak topology on $A$, the Hahn-Banach theorem implies that $a$ belongs to the \emph{norm}-closed convex hull of $\{ \psi_n(\p_n(a)) \}$!!  That is, we can find positive numbers $\theta_1, \ldots, \theta_k$ that sum to one and natural numbers $n_1,\ldots, n_k$ such that 
\[
\xymatrix{ A \ar@{-->}[dr]_{\oplus_{i=1}^k \p_{n_i}} \ar[rr]^{\id_A} & & \, A \\
& \oplus_{i=1}^k \M_{k(n_i)}(\C) \ar@{-->}[ur]_{\sum_{i=1}^k \theta_i \psi_{n_i}} &}
\]
almost commutes on $a$. 

Finally, a standard direct-sum trick allows us to replace individual operators $a \in A$ with finite sets, thereby completing the proof. 
\end{proof} 

The converse of the previous proposition holds too, but it's much harder.  C$^*$-tensor products are a key ingredient in the proof, so the next section is devoted to recalling the necessary definitions and facts.

\section{Tensor Products and The Trick} 

To a von Neumann algebraist there is only one tensor product.  This is a problem.  Indeed, a wonderful feature of the C$^*$-theory is its complexity. This exposes new ideas and sometimes, in the right hands, provides insight that would otherwise remain out of sight. 

\subsection{The spatial and maximal C$^*$-norms}
\label{sec:minmaxnorms}

When $A$ and $B$ are C$^*$-algebras, it can happen that 
numerous different norms make $A\odot B$ (the algebraic tensor product) 
into a pre-C$^*$-algebra. In other words, $A\odot B$ may carry
more than one C$^*$-norm.

\begin{subdefn} A {\em $\mathrm{C}^*$-norm} $\| \cdot \|_{\alpha}$ on $A \odot
B$ is a norm such that $\| xy \|_{\alpha} \leq \| x \|_{\alpha}\|
y \|_{\alpha}$, $\| x^* \|_{\alpha} = \| x \|_{\alpha}$ and $\|
x^*x \|_{\alpha} = \| x \|_{\alpha}^2$ for all $x, y \in A\odot
B$.  We will let $A\otimes_{\alpha} B$ denote the completion of
$A\odot B$ with respect to $\| \cdot \|_{\alpha}$.
\end{subdefn}

It's a fact that C$^*$-norms on algebraic tensor products always exist.  Here are the two
most natural candidates.

\begin{subdefn}(Maximal norm) Given $A$ and $B$, we define the {\em
maximal $\mathrm{C}^*$-norm} on $A\odot B$ to be 
\[
\| x \|_{\max} = \sup \{ \|\pi(x)\| : \pi\colon A\odot B \to \B(\hh)\
\mbox{a $*$-homomorphism}\}
\]
for $x \in A\odot B$. We let $A\otimes_{\max}B$
denote the completion of $A\odot B$ with respect to $\| \cdot
\|_{\max}$.
\end{subdefn}

\begin{subdefn}(Spatial norm) Let $\pi\colon A \to \B(\hh)$ and
$\sigma\colon B \to \B(\hk)$ be {\em faithful} representations.
Then the {\em spatial} (or {\em minimal}) $\mathrm{C}^*$-norm on $A\odot B$
is $$\| \sum a_i \otimes b_i \|_{\min} = \| \sum \pi(a_i)\otimes
\sigma(b_i)\|_{\B(\hh\otimes \hk)}.$$ The completion of $A \odot B$
with respect to $\| \cdot \|_{\min}$ is denoted $A\otimes
B$.\footnote{You will also see $A\otimes_{\min} B$ in the
literature.}
\end{subdefn}

\begin{subrem} 
\label{prop:independent}
It is an important fact that $A \otimes B$ does not depend on the representations $\pi$ and $\sigma$ (see \cite[Proposition 3.3.11]{Me-n-Taka}). 
\end{subrem}

\begin{subrem}[Von Neumann algebra tensor products]
\label{rem:vNtensorproduct}  If $M \subset \B(\hh)$ and $N \subset
\B(\hk)$ are von Neumann algebras, then the {\em von Neumann
algebraic} tensor product $M \vt N$ of $M$ and $N$ is the weak closure of 
$M\otimes N \subset \B(\hh\otimes \hk)$.
\end{subrem}

The following universal property of $\mtp$ is a simple consequence of the definition. 

\begin{subprop}[Universality]
\label{prop:maxuniversality} If $\pi\colon A \odot B \to C$ is a
$*$-homomorphism, then there exists a unique $*$-homomorphism
$A\otimes_{\max}B \to C$ which extends $\pi$.  In particular, a 
pair of $*$-homomorphisms with commuting ranges $\pi_A\colon A \to
C$ and $\pi_B\colon B\to C$ induces a $*$-homomorphism
$$\pi_A \times \pi_B\colon A\otimes_{\max}B \to C.$$
\end{subprop}

It's a remarkable fact that $\| \cdot \|_{\min}$ is really the smallest possible C$^*$-norm on $A\odot B$.   For a proof, see \cite[Section 3.4]{Me-n-Taka}. 

\begin{subthm}[Takesaki]\label{thm:takesaki}
For arbitrary $\mathrm{C}^*$-algebras $A$ and $B$, $\| \cdot \|_{min}$ is the smallest $\mathrm{C}^*$-norm on $A \odot B$.
\end{subthm}

\begin{subcor}
\label{cor:minmax}
For any $A$ and $B$ and any $\mathrm{C}^*$-norm $\|
\cdot \|_{\alpha}$ on $A \odot B$ we have natural surjective
$*$-homomorphisms
$$A\mtp B \to A\otimes_{\alpha} B \to A\otimes B.$$
\end{subcor}

The previous corollary gets used, both explicitly and implicitly, all of the time.  For example, in the literature it is often written that $A\odot B$ has a unique C$^*$-norm if and only if $A\mtp B = A\otimes B$.  

Those tensor products which have a unique C$^*$-norm form an important subclass.  For such examples, the minimal tensor product -- the easier one to get a handle on -- inherits the universal property of maximal tensor products; this is extraordinarily useful, as we'll see later when we come to The Trick.  

Let's record the simplest case of this phenomenon. 

\begin{subprop}
\label{prop:matrix}  For each $\mathrm{C}^*$-algebra $A$ there is a unique 
$\mathrm{C}^*$-norm on the algebraic tensor product $\M_n(\C) \odot A$.
\end{subprop}

\begin{proof}  One checks that there is an algebraic $*$-isomorphism 
$$\M_n(\C) \odot A \cong \M_n(A).$$  Since C$^*$-algebras have unique norms, and $\M_n(A)$ is a C$^*$-algebra, the result follows. \end{proof}

\subsection{Continuity of tensor product maps}
\label{sec:tensorproductmaps}

The tensor product of bounded linear maps need not be bounded. But c.p.\ maps are nice.  

\begin{subthm}[Continuity of tensor product maps]\label{thm:tenprodmapscont}
Let $\p\colon A \to C$ and $\psi\colon B \to D$ be c.p.\ maps.  
Then the algebraic tensor product map
$$\p\odot \psi\colon A\odot B \to C\odot D$$ extends to a c.p.\
(hence continuous) map on both the minimal and maximal tensor
products. Moreover, letting $\p\otimes_{\max}\psi\colon
A\otimes_{\max} B \to C\otimes_{\max} D$ and $\p\otimes \psi\colon
A\otimes B \to C\otimes D$ denote the extensions, we have
$$\|\p\otimes_{\max} \psi \| = \|\p\otimes \psi \| = \|\p \|\|
\psi \|.$$
\end{subthm}

\begin{proof}  Both cases are consequences of Stinespring's Dilation Theorem. With that result in hand, the minimal case is a routine exercise; the maximal case is only harder because one needs to know Arveson's `commutant lifting' version of Stinespring's result.  See \cite[Theorem 3.5.3]{Me-n-Taka} for details. 
\end{proof}

The following corollary will be used frequently and without reference.

\begin{subcor}  Assume $\theta\colon A \to C$ and $\sigma\colon B\to
D$ are c.c.p.\ maps and $\theta_n\colon A\to C$ are c.c.p.\ maps
converging to $\theta$ in the point-norm topology (i.e., $\| \theta_n(a) - \theta(a)\| \to 0$ for all $a \in A$).  Then
$$\theta_n \mtp \sigma \to \theta \mtp \sigma$$ and $$\theta_n
\otimes \sigma \to \theta\otimes \sigma$$ in the point-norm
topology as well.
\end{subcor}

Generalizing the matrix case (Proposition \ref{prop:matrix}), our next result is extremely important. 

\begin{subprop} 
\label{prop:tensornuc}  
If $A$ is nuclear, then for every $B$ there is a unique C$^*$-norm on $A \odot B$. 
\end{subprop} 

\begin{proof}  It suffices to show that if $x \in A\odot B$, then $\| x \|_{\max} \leq \| x \|_{\min}$, since this implies the canonical quotient mapping $A \mtp B \to A\otimes B$ is isometric on a dense set. 

So, let $x \in A\odot B$ be given.  Since $A$ is nuclear, there are c.c.p.\ maps $\p_n \colon A \to \M_{k(n)}(\C)$ and $\psi_n \colon \M_{k(n)}(\C) \to A$ converging to $\id_A$ in the point-norm topology.  Thus we can define c.c.p.\ maps $\theta_n \colon A\otimes B \to A \mtp B$ by $$\theta_n = (\psi_n \mtp \id_B) \circ (\p_n \otimes \id_B),$$ where we've used the identification $\M_{k(n)}(\C) \mtp B = \M_{k(n)}(\C) \otimes B$ to make sense of the composition. Evidently $\| x - \theta_n(x) \|_{\max} \to 0$ and hence $\| x \|_{\max} = \lim \| \theta_n(x)\| \leq \| x \|_{\min}$, as desired. 
\end{proof} 

\subsection{Inclusions and The Trick}
\label{sec:inclusions}

C$^*$-tensor products can be subtle; they don't always behave like algebraic tensor products.  Let's have a look at an important subtlety, as well as The Trick to which it leads. 

The issue is whether or not inclusions of C$^*$-algebras
give rise to inclusions of tensor products.  For algebraic tensor products this is always the case, hence  {\em spatial} tensor products are also kind and inclusive. 

\begin{subprop}
\label{prop:mininclusion} If $A \subset B$ and $C$ are
$\mathrm{C}^*$-algebras, then there is a natural inclusion $$A\otimes C
\subset B\otimes C.$$
\end{subprop}

\begin{proof} Perhaps we should first point out what this
proposition is really asserting.  Since we have a natural
algebraic inclusion $$A\odot C \subset B\odot C,$$ one can ask which norm we get on $A\odot C$ by restricting the spatial
norm on $B\odot C$.  This proposition asserts
that we just get the spatial norm on $A\odot C$.

Having understood the meaning of the result, there is nothing to prove since we can choose faithful representations of $B$ and $C$ to construct $A\otimes C$ (see Remark \ref{prop:independent}).
\end{proof}

Similarly a pair of inclusions $A \subset B$ and $C \subset D$ gives rise to an inclusion $A\otimes C \subset B\otimes D$.

For maximal tensor products this inclusion business doesn't always work, which may
seem a little puzzling at first.  However, when reformulated at
the algebraic level, it becomes clear what can go wrong.  Indeed, 
what we are really asking is whether or not the maximal norm on $B
\odot C$ restricts to the maximal norm on $A\odot C \subset B\odot
C$. But the maximal norm is defined via a supremum over
representations and since every representation of $B\odot C$ gives
a representation of the smaller algebra $A\odot C$, it is clear
that the supremum only over representations of $B\odot C$ will
always be less than or equal to the supremum over all
representations of $A\odot C$.

Having seen what the problem could be, here's a case where everything goes well. 

\begin{subprop}
\label{cor:nuclearWEP} If $A \subset B$, $A$ is nuclear and $C$ is arbitrary, then we have a natural inclusion $A\mtp C \subset B\mtp C.$
\end{subprop}

\begin{proof}  Since $A\mtp C = A\otimes C \subset B\otimes C$, by Propositions \ref{prop:tensornuc} and \ref{prop:mininclusion}, it follows that the canonical $*$-homomorphism $A\mtp C \to B\mtp C \to B\otimes C$ can't have a nontrivial kernel. Thus $A\mtp C \to B\mtp C$ must be injective. 
\end{proof} 

Here's a more general result. 

\begin{subprop}
\label{prop:relwklyinj} Let $A \subset B$ be an inclusion of
$\mathrm{C}^*$-algebras and assume that for every nondegenerate
$*$-homomorphism $\pi\colon A\to \B(\hh)$ there exists a c.c.p.\
map $\p\colon B\to \pi(A)^{\prime\prime}$ such that $\p(a) =
\pi(a)$ for all $a \in A$.  Then for every $\mathrm{C}^*$-algebra $C$ there
is a natural inclusion $$A\mtp C \subset B\mtp C.$$
\end{subprop}

\begin{proof}  We must  show that if  $x \in A\mtp C$ is in the kernel of the canonical map $A\mtp C \to B\mtp C$, then $x = 0$.

Let $\pi\colon A\mtp C \to \B(\hh)$ be a faithful
representation.  It's a fact, nontrivial only in the nonunital case (see \cite[Theorem 3.2.6]{Me-n-Taka}), that one can find $*$-homomorphisms $\pi_A\colon A\to \B(\hh)$ and $\pi_C\colon C\to \B(\hh)$ (called the \emph{restrictions}) with commuting ranges such that $\pi(a\otimes c) = \pi_A(a) \pi_C(c)$ for all $a \in A$ and $c \in C$.  Since $\pi_C(C) \subset
\pi_A(A)^{\prime}$,  the commuting inclusions
$\pi_A(A)^{\prime\prime} \hookrightarrow \B(\hh)$, $\pi_C(C)
\hookrightarrow \B(\hh)$ induce, by universality, a product
$*$-homomorphism
$$\pi_A(A)^{\prime\prime} \mtp \pi_C(C) \longrightarrow \B(\hh).$$

Extend $\pi_A$ to a c.c.p.\ map $\p\colon B\to
\pi(A)^{\prime\prime}$ such that $\p(a) = \pi_A(a)$ for all $a \in
A$.  By Theorem \ref{thm:tenprodmapscont} we have the following
commutative diagram: 
\[
\xymatrix{ B\mtp C  \ar[rr]^{\p \mtp \pi_C} & & \,
\pi_A(A)^{\prime\prime} \mtp \pi_C(C) \ar[d] \\
A\mtp C \ar[u], \ar[rr]^{\pi} & & \, \B(\hh).}
\]
The fact that $\pi$ is faithful implies that the map on the left
is also injective.
\end{proof}

The converse of the previous result holds too. But to prove it, we need The Trick -- arguably the most useful observation about C$^*$-tensor products ever made. 

\begin{subprop}[The Trick]\label{prop:thetrick}
Let $A\subset B$ and $C$ be $\mathrm{C}^*$-algebras, $\|\cdot \|_{\alpha}$ 
be a $\mathrm{C}^*$-norm on $B\odot C$ and $\|\cdot \|_{\beta}$ be the
$\mathrm{C}^*$-norm on $A\odot C$ obtained by restricting $\|\cdot
\|_{\alpha}$ to $A\odot C \subset B\odot C$. If $\pi_A\colon A\to
\B(\hh)$, $\pi_C\colon C \to \B(\hh)$ are representations with
commuting ranges and if the product $*$-homomorphism $$\pi_A
\times \pi_C\colon A\odot C \to \B(\hh)$$ is continuous
with respect to $\| \cdot \|_{\beta}$, then there exists a c.c.p.\
map $\p\colon B\to \pi_C(C)^{\prime}$ which extends $\pi_A$.
\end{subprop}

\begin{proof}  To avoid annoying details, we will assume  that $A$, $B$ and $C$ are all unital
and, moreover, that $1_A = 1_B$ (see \cite[Proposition 3.6.5]{Me-n-Taka} for the general case). Let $$\pi_A \times_{\beta}
\pi_C\colon A\otimes_{\beta} C \to \B(\hh)$$ be the extension of
the product map to $A\otimes_{\beta} C$.  Since $A\otimes_{\beta}
C \subset B\otimes_{\alpha} C$, we apply Arveson's Extension
Theorem (\cite[Theorem 1.6.1]{Me-n-Taka}) to get a u.c.p.\ extension $\Phi\colon B\otimes_{\alpha} C
\to \B(\hh)$. The desired map is just $\p(b) = \Phi(b\otimes
1_C)$.

To see that $\p$ takes values in $\pi_C(C)^{\prime}$ is a simple
multiplicative domain argument.\footnote{Multiplicative domains come up several times in these notes.  If $\psi \colon A \to B$ is a c.c.p.\ map, then the multiplicative domain of $\psi$ is the C$^*$-subalgebra $$A_\psi = \{ a\in A: \psi(aa^*) = \psi(a)\psi(a^*) \mbox{ and } \psi(a^*a) = \psi(a^*)\psi(a) \}.$$ This is precisely the elements that commute with the Stinespring projection and hence one has $\psi(ab) = \psi(a)\psi(b)$ and $\psi(ba) = \psi(b)\psi(a)$ for all $a \in A_\psi$ and all $b \in A$ (cf.\ \cite[Proposition 1.5.7]{Me-n-Taka}).}  Indeed, $\C 1_B\otimes C$ lives in the multiplicative domain of $\Phi$ since $\Phi|_{\C 1_B\otimes C} =
\pi_C$ is a $*$-homomorphism.  Thus, for every $b \in B$ and $c \in C$ we have 
\begin{align*} 
\p(b)\pi_C(c) &= \Phi(b\otimes 1_C) \Phi(1_B \otimes c)\\ 
& = \Phi((b \otimes 1_C)(1_B \otimes c))\\ 
& = \Phi((1_B \otimes c)(b \otimes 1_C))\\ 
& = \pi_C(c)\p(b) 
\end{align*}
\end{proof}

The Trick is hard to appreciate until you see what it can do for you.  But before demonstrating its utility, let me emphasize the point. Given an inclusion $A \subset B$ and a representation $\pi\colon
A\to \B(\hh)$, Arveson's Extension Theorem always allows one to
extend $\pi$ to a c.c.p.\ map $\p\colon B \to \B(\hh)$.  When The
Trick is applicable, \emph{one has better control on the range of this extension}; the point of The Trick is  that $\p(B) \subset \pi_C(C)^{\prime}$. 

As a first application, let's prove the converse of Proposition \ref{prop:relwklyinj}. An inclusion satisfying one of the following equivalent conditions is called {\em relatively weakly injective}.

\begin{subprop}\label{prop:relwklyinjII}
Let $A \subset B$ be an inclusion. Then
the following are equivalent:
\begin{enumerate}
\item there exists a c.c.p.\ map $\p\colon B \to A^{**}$ such that
$\p(a) = a$ for all $a \in A$;

\item for every $*$-homomorphism $\pi\colon A\to \B(\hh)$ there
exists a c.c.p.\ map $\p\colon B\to \pi(A)^{\prime\prime}$ such
that $\p(a) = \pi(a)$ for all $a \in A$;

\item for every $\mathrm{C}^*$-algebra $C$ there is a natural inclusion
$$A\mtp C \subset B\mtp C.$$
\end{enumerate}
\end{subprop}

\begin{proof}  Since every representation of $A$ extends to a
normal representation of $A^{**}$, the equivalence of the first two
statements is easy.

Assume condition (3) and let $\pi\colon A\to \B(\hh)$ be a
representation.  Let $C = \pi(A)^{\prime}$ and, by universality,
we can apply The Trick to the product map induced by the commuting
representations $\pi\colon A\to \B(\hh)$ and $\pi(A)^{\prime}
\hookrightarrow \B(\hh)$.  That's it.
\end{proof}

Our second application of The Trick is just as easy.  Recall that a von Neumann algebra $M \subset \B(\hh)$ is called \emph{injective} if there exists a conditional expectation $\B(\hh) \to M$ -- i.e., a u.c.p.\ map $\Phi\colon \B(\hh) \to M$ such that $\Phi(x) = x$ for all $x \in M$. 

\begin{subprop}  
\label{prop:comminj}
If $A$ is nuclear and $\pi\colon A \to \B(\hh)$ is a representation, then $\pi(A)^{\prime}$ is injective. 
\end{subprop} 

\begin{proof}  The canonical $*$-homomorphism $\iota \times \pi\colon \pi(A)^{\prime} \odot A \to \B(\hh)$, where $\iota\colon \pi(A)^{\prime}  \to \B(\hh)$ is inclusion, extends to $\pi(A)^{\prime} \mtp A$ by universality; i.e., by Proposition \ref{prop:tensornuc}, it extends to $\pi(A)^{\prime} \otimes A$.  Thus we can apply The Trick to  $\pi(A)^{\prime} \otimes A \subset \B(\hh) \otimes A$, and we're done. 
\end{proof} 

Though it won't be needed, let's look at one more application before getting back to nuclearity and von Neumann algebras.  I can't overstate the power and importance of The Trick, so please forgive the digression. 

\begin{subdefn} 
\label{defn:wep}
A $\mathrm{C}^*$-algebra $A\subset\B(\hh)$ is said to have Lance's \emph{weak expectation property} (WEP) if there exists a u.c.p.\ map $\Phi\colon \B(\hh) \to A^{**}$ such that $\Phi(a) = a$ for all $a \in A$.
\end{subdefn} 

A simple application of Arveson's Extension Theorem shows that the WEP is independent of the choice of faithful representation.

\begin{subcor}
\label{cor:WEPifftensorincl}  A $\mathrm{C}^*$-algebra $A$ has the WEP if
and only if for every inclusion $A\subset B$ and arbitrary $C$ we
have a natural inclusion $A\mtp C \subset B\mtp C$.
\end{subcor}

\begin{proof} Assume first that $A\subset\B(\hh)$ has the WEP and $A \subset B$.
The inclusion $A\hookrightarrow\B(\hh)$ extends to a c.c.p.\ map $\Psi\colon B \to
\B(\hh)$ by Arveson's Extension Theorem. Composing with $\Phi$ gives a map $B \to A^{**}$
which restricts to the identity on $A$ and then Proposition \ref{prop:relwklyinjII}
applies. The converse uses The Trick just as in the previous proposition.
This time take $B = \B(\hh_{\mathcal{U}})$, the universal representation of $A$, 
and $C=(A^{**})^{\prime}$.
\end{proof}

\section{Nuclearity and Injectivity} 
 
Perhaps I should remind you of the first goal of these notes: To describe an important C$^*$-theorem which requires a W$^*$-proof. This section contains such a theorem.  However, I think it's very instructive to specialize to the case of group C$^*$-algebras before getting into the general case.  Moreover, we can essentially give a complete proof in the group case (the W$^*$-machinery required by the general case is far too hard and long to include in these notes). 
 
\subsection{Reduced Group C$^*$-algebras} 

Mostly to establish notation, let me quickly review some basics. For a discrete group $\G$ we let $\lambda\colon \G \to \B(\ell^2(\G))$ denote the {\em left regular representation}: 
$\lambda_s(\delta_t) = \delta_{st}$ for all $s,t \in \G$, where
$\{\delta_t:t \in \G\}\subset \ell^2(\Gamma)$ is the canonical
orthonormal basis.  We'll also need the \emph{right regular representation} $\rho\colon \G 
\to \B(\ell^2(\G))$, defined by $\rho_s(\delta_t) = \delta_{ts^{-1}}$. 
Note that $\lambda$ and $\rho$ are unitarily equivalent; the intertwining unitary is 
defined by $U\delta_t=\delta_{t^{-1}}$. Also, note that the left and right regular representations commute, i.e., $\lambda_s \rho_t = \rho_t \lambda_s$ for all $s,t \in \G$. 

There is a canonical left action of $\G$ on $\ell^{\infty}(\Gamma)$.  For $f \in
\ell^{\infty}(\Gamma)$ and $s \in \G$ we let $s.f \in
\ell^{\infty}(\Gamma)$ be the function $s.f(t) = f(s^{-1}t)$; simple calculations
show that $f \mapsto s.f$ defines a group action of $\Gamma$
on $\ell^{\infty}(\Gamma)$. 
This action is {\em spatially
implemented by the left regular representation}  -- a very important fact. That is, if we
regard $\ell^{\infty}(\Gamma) \subset \B(\ell^2(\G))$ as
multiplication operators (i.e., $f\delta_t = f(t)\delta_t$), then a
calculation shows $$\lambda_s f \lambda_s^* = s.f$$
for all $f \in \ell^{\infty}(\Gamma)$ and $s \in \G$.

The \emph{reduced} C$^*$-algebra of $\G$, denoted $C_{\lambda}^*(\G)$,\footnote{
You will also see $C^*_r(\G)$ in the literature.}
is the C$^*$-algebra generated by $\{ \lambda_s : s \in \G \}$. In these notes it will be important to distinguish this algebra from $C_{\rho}^*(\G)$, which is just 
the C$^*$-algebra generated by the right regular representation (even though these algebras are isomorphic). Since the left and right regular representations commute, we have a canonical (and very important!) $*$-homomorphism $$C_{\lambda}^*(\G) \odot C_{\rho}^*(\G) \to \B(\ell^2(\G)).$$

The {\em group von Neumann algebra of $\G$} is defined to be
$$L(\G) := C_{\lambda}^*(\G)^{\prime\prime} \subset \B(\ell^2(\G)).$$
A fundamental theorem of Murray and von Neumann states that $L(\G)$
is the commutant of the right regular representation -- i.e., $L(\G) =
C_{\rho}^*(\G)^{\prime}$ and $L(\G)^{\prime} =
C_{\rho}^*(\G)^{\prime\prime}$.  As is well known, the vector state $T \mapsto \langle T \delta_e, \delta_e\rangle$, where $e \in \G$ denotes the neutral element, defines a faithful trace on $L(\G)$, which we will denote by $\tau$.

\begin{subdefn}\label{defn:amenablegrp}
A group $\G$ is \emph{amenable} if there exists a state $\mu$ on
$\ell^\infty(\G)$ which is invariant under the left translation action, i.e., $\mu(s.f) = \mu(f)$ 
 for all $s \in \G$ and $f \in \ell^\infty(\G)$.

Such a state $\mu$ is called an \emph{invariant mean}.
\end{subdefn}

Here's a theoretically useful way to get an invariant mean.  

\begin{sublem} 
\label{lem:invmean} Assume there exists a u.c.p.\ map $\Phi\colon \B(\ell^2(\G)) \to L(\G)$ such that $\Phi(x) = x$ for all $x \in C^*_{\lambda}(\G)$.  Then $\G$ is amenable. 
\end{sublem} 

\begin{proof} The state $\mu := \tau\circ \Phi|_{\ell^{\infty}(\G)}$ turns out to be an invariant mean.  Indeed, since $C^*_{\lambda}(\G)$ falls in the multiplicative domain of $\Phi$ and the left translation action is spatially implemented, we have $$\mu(s.f) = \tau( \Phi(\lambda_s f \lambda^*_s)) = \tau( \lambda_s \Phi(f) \lambda^*_s) = \tau(\Phi(f)) = \mu(f),$$ for all $f \in \ell^{\infty}(\G)$ and $s \in \G$. 
\end{proof}

The \emph{symmetric difference} of two sets $E$ and $F$, 
denoted $E\bigtriangleup F$, is $(E\cup F) \setminus (E\cap F)$. 

\begin{subdefn}\label{defn:Folner}
We say $\G$ satisfies the \emph{F{\o}lner condition} if for any
finite subset $E\subset\G$ and $\e>0$, there exists a finite
subset $F\subset\G$ such that
\[
\max_{s\in E}\frac{|sF\bigtriangleup F|}{|F|}< \e,
\]
where $sF=\{st :
t\in F\}$.\footnote{Since $sF\bigtriangleup F =
[sF\setminus(sF\cap F)] \cup [F\setminus (sF\cap F)]$, it follows
that $\frac{|sF\bigtriangleup F|}{|F|} = 2 - 2\frac{|F\cap
sF|}{|F|}$. Hence the F{\o}lner condition is equivalent to
requiring $\max_{s\in E}\frac{|sF\cap F|}{|F|}> 1 - \e/2$, which
is often how it gets used in our context.} A sequence of finite 
sets $F_n \subset \G$ such that $$\frac{|sF_n\bigtriangleup F_n|}{|F_n|} \to 0$$ 
for every $s \in \G$ is called a \emph{F{\o}lner sequence}. 
\end{subdefn}

In the context of discrete groups, the following result illustrates the connections between nuclearity and von Neumann algebras.

\begin{subthm}
\label{thm:amenablegroup}
Let $\G$ be a discrete group. The following are equivalent:
\begin{enumerate}
\item\label{thm:grpam1} $\G$ is amenable;


\item\label{thm:grpam3} $\G$ satisfies the F{\o}lner condition;






\item\label{thm:grpam9} $C_{\lambda}^*(\G)$ is nuclear;

\item\label{thm:grpam9.5} the canonical $*$-homomorphism $C_{\lambda}^*(\G) \odot C_{\rho}^*(\G) \to \B(\ell^2(\G))$ is $\min$-continuous (i.e., extends to $C_{\lambda}^*(\G) \otimes C_{\rho}^*(\G)$); 

\item\label{thm:grpam10} $L(\G)$ is semidiscrete;

\item\label{thm:grpam11} $L(\G)$ is injective. 
\end{enumerate}
\end{subthm}

\begin{proof} The equivalence of (1) and (2) is classical (see \cite[Theorem 2.6.8]{Me-n-Taka} for details).  

$(\ref{thm:grpam3})\Rightarrow(\ref{thm:grpam9})$:
Let $F_k \subset \Gamma$ be a sequence
of F{\o}lner sets. For each $k$ let $P_k \in
\B(\ell^2(\Gamma))$ be the orthogonal projection onto the finite-dimensional subspace spanned by $\{ \delta_g:g \in F_k\}$. Identify
$P_k \B(\ell^2(\Gamma))P_k$ with the
 matrix algebra $\M_{F_k}(\C)$ and let $\{
e_{p,q} \}_{p,q \in F_k}$ be the canonical matrix units of $\M_{F_k}(\C)$.  One can check that for each $s \in \Gamma$ we have
$e_{p,p}\lambda_s e_{q,q} = 0$ unless $sq = p$, and
$e_{p,p}\lambda_s e_{q,q} = e_{p,q}$ if $sq = p$. Since $P_k =
\sum_{p \in F_k} e_{p,p}$, we have $$P_k \lambda_s P_k = \sum_{p,q
\in F_k} e_{p,p}\lambda_s e_{q,q} = \sum_{p \in F_k \cap sF_k}
e_{p,s^{-1}p}.$$ Let $\p_k\colon C^*_{\lambda}(\G) \to
\M_{F_k}(\C)$ be the u.c.p.\ map defined by $x \mapsto P_k
xP_k$. Now define a map $\psi_k\colon \M_{F_k}(\C) \to C^*_{\lambda}(\G)$
by sending $$e_{p,q} \mapsto \frac{1}{|F_k|}
\lambda_p\lambda_{q^{-1}}.$$ Evidently this map is unital; it is also completely
positive, as one can check. 

The $\p_k$'s and $\psi_k$'s do the trick. Since the linear span of
$\{\lambda_s:s \in \Gamma\}$ is norm dense in $C^*_{\lambda}(\G)$, it
suffices to check that $\| \lambda_s - \psi_k\circ\p_k(\lambda_s)
\|\to 0$ for all $s \in \Gamma$. This follows from the
definition of F{\o}lner sets together with the following
computation:
$$\psi_k\circ\p_k(\lambda_s) = \psi_k( \sum_{p \in F_k \cap sF_k}
e_{p,s^{-1}p}) = \sum_{p \in F_k \cap sF_k}
\frac{1}{|F_k|}\lambda_s = \frac{|F_k \cap
sF_k|}{|F_k|}\lambda_s.$$ Hence the reduced group C$^*$-algebra is nuclear.

$(\ref{thm:grpam9})\Rightarrow(\ref{thm:grpam9.5})$:  This follows immediately from Proposition \ref{prop:tensornuc}, since $C_{\lambda}^*(\G) \odot C_{\rho}^*(\G) \to \B(\ell^2(\G))$ is always $\max$-continuous. 

$(\ref{thm:grpam9.5})\Rightarrow(\ref{thm:grpam1})$:  Applying The Trick to $C_{\lambda}^*(\G) \otimes C_{\rho}^*(\G) \subset L(\G) \otimes C_{\rho}^*(\G)$ we get a u.c.p.\ map $\Phi\colon \B(\ell^2(\G)) \to L(\G)$ such that $\Phi(x) = x$ for all $x \in C^*_{\lambda}(\G)$. Hence, Lemma \ref{lem:invmean} implies $\G$ is amenable. 

$(\ref{thm:grpam3})\Rightarrow(\ref{thm:grpam10})$ is similar to the proof of $(\ref{thm:grpam3})\Rightarrow(\ref{thm:grpam9})$. 

$(\ref{thm:grpam10})\Rightarrow(\ref{thm:grpam11})$: Let $\p_n\colon L(\G) \to \M_{k(n)}(\C)$ and $\psi_n\colon \M_{k(n)}(\C) \to L(\G)$ be as in the definition of semidiscreteness.  By Arveson's Extension Theorem, we may assume the $\p_n$'s are defined on all of $\B(\ell^2(\G))$. By Theorem \ref{subthm:pointultraweak}, we can find a point-ultraweak cluster point of the maps $\psi_n \circ \p_n \colon \B(\ell^2(\G)) \to L(\G)$ and this is evidently a conditional expectation. 

$(\ref{thm:grpam11})\Rightarrow(\ref{thm:grpam1})$ is an immediate consequence of Lemma \ref{lem:invmean}.
\end{proof}

As mentioned earlier, one of the reasons I've taken the time to prove the theorem above is that it isn't all that hard.  And it suggests there might be a more general result lurking in the bushes.

\subsection{W$^*$-algebras and the general case} 

In the proof of Theorem \ref{thm:amenablegroup} we used the group $\G$ in every implication -- except when proving $(\ref{thm:grpam9})\Rightarrow(\ref{thm:grpam9.5})$.  This implication followed from the general fact that if $A$ is nuclear, then there is a unique C$^*$-norm on $A\odot B$ for every $B$ (Proposition \ref{prop:tensornuc}). However, we've seen that  nuclear C$^*$-algebras have injective commutants in any representation (Proposition \ref{prop:comminj}).  Thanks to Haagerup standard form (\cite{haagerup}), we can take another step. 

\begin{subprop} 
\label{prop:nucinj} If $A$ is nuclear, then $A^{**}$ is injective. 
\end{subprop} 

\begin{proof} First we represent $A^{**} \subset \B(\hh)$ in standard form. Then, there is a conjugate-linear isometry $J\colon \hh \to \hh$ such that $J^2 = \id_{\hh}$ and $A^{**} = J(A^{**})^{\prime} J$.  By Proposition \ref{prop:comminj}, there is a conditional expectation $\Phi\colon \B(\hh) \to (A^{**})^{\prime}$.  Now check that $T \mapsto J\Phi(JTJ)J$ is a conditional expectation onto $A^{**}$, and we're done. 
\end{proof} 

Here's the big theorem we've been after. 

\begin{subthm}  
\label{thm:nucchar}
For a C$^*$-algebra $A$, the following are equivalent: 
\begin{enumerate} 
\item for every $B$ there is a unique C$^*$-norm on $A\odot B$;\footnote{Originally this was the definition of nuclearity.} 

\item $A$ is nuclear; 

\item $A^{**}$ is semidiscrete; 

\item $A^{**}$ is injective.  
\end{enumerate} 
\end{subthm} 

\begin{proof} We've already seen $(3) \Longrightarrow (2)$ (Proposition \ref{prop:seminuc}), $(2) \Longrightarrow (1)$ (Proposition \ref{prop:tensornuc}) and $(1) \Longrightarrow (4)$ (Proposition \ref{prop:nucinj}). 

The remaining implication, mainly due to Connes and Choi-Effros, is very hard work -- too hard for these notes. It depends on some deep results in von Neumann algebra theory (e.g., existence of modular automorphisms) and, in my opinion, is one of the great W$^*$-achievements.  See \cite[Section 9.3]{Me-n-Taka} for more on the proof. \end{proof} 

At present there is no proof of the equivalence of conditions (1) and (2) above -- a purely C$^*$-algebraic statement! -- that avoids von Neumann algebras. 

If $J \triangleleft A$ is a closed 2-sided ideal, then $A^{**} = J^{**} \oplus (A/J)^{**}$. Hence the next corollary is easily deduced from the previous theorem. 

\begin{subcor} 
\label{cor:nucquot}
Let $0 \to I \to A \to A/I \to 0$ be a short exact sequence.  Then $A$ is nuclear if and only if both $I$ and $A/I$ are nuclear. 
\end{subcor} 

To the C$^*$-algebraist that doesn't yet appreciate von Neumann algebras, I have a challenge: Find a C$^*$-proof of the fact that nuclearity passes to quotients.  Good luck... 

Here's another nice C$^*$-application (due to Blackadar). It requires the fact that nuclearity passes to quotients, hence von Neumann algebras deserve much of the credit.  

\begin{subcor}
\label{cor:nuctypeI}
A separable $\mathrm{C}^*$-algebra is type \emph{I} if and only if every subalgebra
is nuclear.
\end{subcor}

\begin{proof}  Recall Glimm's Theorem: A separable C$^*$-algebra
$A$ is \emph{not} type I if and only if every UHF algebra arises
as a subquotient of $A$ (see \cite[Section 6.8]{pedersen}).

This implies, first of all, that subalgebras of type I are again type I.
Indeed, if a subalgebra were not type I, then it would have UHF
subquotients and hence the larger algebra would too.  Since type
I C$^*$-algebras are nuclear (\cite[Proposition 2.7.4]{Me-n-Taka}),
this evidently implies the ``only if" direction.

For the opposite direction, assume $A$ is not type I.  By Glimm's Theorem, we
can find a subalgebra $B \subset A$ with a prescribed
UHF quotient.  It is a fact that  $C^*_\lambda(\F_2)$ is a subquotient of every UHF algebra (cf.\ \cite[Corollary 8.2.5]{Me-n-Taka}), hence we can find $C \subset B$ which has $C^*_\lambda(\F_2)$ as a quotient.  Since $C^*_\lambda(\F_2)$ is not nuclear (cf.\ Theorem \ref{thm:amenablegroup}), $C$ is not nuclear. 
\end{proof}

\section{Solid von Neumann Algebras} 

Now that we've seen a few C$^*$-theorems that require W$^*$-proofs, I want to turn the tables and show you a few C$^*$-contributions to von Neumann algebra theory.

\subsection{Exact C$^*$-algebras and local reflexivity} 

Here's a horribly short summary of some C$^*$-facts we'll need. 

\begin{subdefn} A C$^*$-algebra $A$ is \emph{exact} if there exists a faithful, nuclear (see Definition \ref{defn:nuclear}) $*$-representation $\pi\colon A \to \B(\hh)$.\footnote{As with nuclearity, this is not the historically correct definition. A deep theorem of Kirchberg states that the original definition (which involved tensor products and short exact sequences) is equivalent to the definition above (see \cite[Theorem 3.9.1]{Me-n-Taka}).}  A discrete group $\G$ is exact if $C^*_{\lambda}(\G)$ is exact. 
\end{subdefn}

It turns out that exact C$^*$-algebras enjoy an important approximation property that isn't as well known as it should be. 

\begin{subdefn}\label{defn:locref}
A $\mathrm{C}^*$-algebra $A$ is {\em locally reflexive} if for every finite-dimensional operator system\footnote{That is, $E$ is a self-adjoint linear subspace containing the unit of $A^{**}$.} $E\subset A^{**}$, there exists a net
of c.c.p.\ maps $\p_i\colon E\to A$ which converges to $\id_E$ in
the point-ultraweak topology.
\end{subdefn}

In a brutal tour de force, Kirchberg proved the following remarkable result (see \cite[Chapter 9]{Me-n-Taka}).

\begin{subthm}\label{cor:exactlocref}
Exact $\mathrm{C}^*$-algebras are locally reflexive.\footnote{This result depends on some older tensor product work of Archbold and Batty, which, in turn, depends on Theorem \ref{thm:nucchar}.   As such, it is another example of a C$^*$-theorem that requires von Neumann algebras. But we're about to come full circle and use this result to prove Ozawa's solidity theorem for von Neumann algebras, so it's a bit misleading to say that Ozawa's work is an application of C$^*$-theory to W$^*$-algebras.  It's really a brilliant reminder of the unity of Operator Algebras!}
\end{subthm}

\subsection{Solid von Neumann algebras}

In hopes of making Ozawa's work easier to digest, I'll cut it into bite-size chunks.  The first bite is a tasty application of local reflexivity. 

\begin{sublem}\label{lem:lem} 
Let $M \subset \B(\hh)$ be a von Neumann algebra which contains a weakly dense exact C$^*$-algebra $B \subset M$.  Assume $N \subset M$ is a von Neumann subalgebra with a weakly continuous conditional expectation $\Phi\colon M \to N$, such that there exists a u.c.p.\ map $\Psi\colon \B(\hh) \to M$ with the property that $\Phi|_B = \Psi|_B$.  Then $N$ is injective. 
\end{sublem} 

\begin{proof} Let $E \subset N$ be a finite-dimensional operator system and $\p_n \colon E \to B$ be contractive c.p.\ maps converging to $\id_E$ in the point-ultraweak topology.  By Arveson's Extension Theorem, we may assume each $\p_n$ is defined on all of $\B(\hh)$ (and now takes values in $\B(\hh)$). Then one readily checks that $\Phi\circ\Psi \circ \p_n \colon \B(\hh) \to N$ are u.c.p.\ maps with the property that $\Phi\circ\Psi \circ \p_n(x) \to x$ ultraweakly for all $x \in E$ (since $\Phi\circ\Psi \circ \p_n(x) = \Phi(\p_n(x))$ for all $x \in E$).  Taking a cluster point in the point-ultraweak topology we get a u.c.p.\ map $\theta_E\colon \B(\hh) \to N$ which restricts to the identity on $E$.  Taking another cluster point of the maps $\theta_E$ (over all finite-dimensional operator systems $E \subset N$) we get the desired conditional expectation $\B(\hh) \to N$. (I love Theorem \ref{subthm:pointultraweak}!) 
\end{proof} 

Now, let's specialize to the case of group von Neumann algebras.  Another spectacular feature of finite von Neumann algebras is the existence of conditional expectations; that is, if we had assumed $M$ to be finite in the previous proposition, then the existence of a conditional expectation $M \to N$ would have been automatic (see \cite[Lemma 1.5.11]{Me-n-Taka}).

\begin{sublem}\label{lem:thelem} Assume $\G$ is exact and let $N \subset L(\G)$ be a von Neumann sublagebra with trace-preserving conditional expectation $\Phi\colon L(\G) \to N$. If $$\Phi|_{C^*_\lambda(\G)} \times \id_{C^*_{\rho}(\G)} \colon C^*_\lambda(\G) \odot C^*_{\rho}(\G) \to \B(\ell^2(\G))$$ is $\min$-continuous, then $N$ is injective. 
\end{sublem} 

\begin{proof} Since $C^*_\lambda(\G) \otimes C^*_{\rho}(\G) \subset \B(\ell^2(\G)) \otimes C^*_{\rho}(\G)$, The Trick applied to $\Phi|_{C^*_\lambda(\G)} \times \id_{C^*_{\rho}(\G)}$ yields a u.c.p.\ map $\Psi\colon \B(\ell^2(\G)) \to L(\G)$ such that $\Psi|_{C^*_\lambda(\G)} = \Phi|_{C^*_\lambda(\G)}$.  Hence Lemma \ref{lem:lem} applies. 
\end{proof} 

\begin{subdefn} A von Neuman algebra $M$ is called \emph{solid} if the relative commutant of every diffuse\footnote{I.e., has no minimal projections.} von Neumann subalgebra is injective. 
\end{subdefn} 

When I first saw this definition, I thought, ``Huh? What's that all about?" Well, it turns out to be a very strong structural statement.  For example, solidity passes to subalgebras and implies that nontrivial tensor product decompositions are impossible (for non-injective algebras). 

\begin{subprop} 
\label{prop:prime} 
If $M$ is finite and solid, and $N \subset M$ is a non-injective subfactor, then $N$ is prime (i.e., $N$ is not isomorphic to the tensor product of II$_1$-factors).  
\end{subprop} 

\begin{proof} Assume $N$ is not prime.  Then we can write $N = N_1 \bar{\otimes} N_2$ where $N_1$ is diffuse and $N_2$ is not injective. By definition, the relative commutant of $N_1$ (in $M$) is injective.  But this commutant contains $N_2$, which is a contradiction (since there is a conditional expectation $N_1^\prime \cap M \to N_2$, hence injectivity would pass to $N_2$). 
\end{proof} 

Here's Ozawa's celebrated solidity theorem. 

\begin{subthm} 
\label{thm:taka}
Assume $\G$ is exact and the canonical map $$C^*_\lambda(\G) \odot C^*_{\rho}(\G) \to \B(\ell^2(\G)) \to \B(\ell^2(\G))/\K(\ell^2(\G))$$ is $\min$-continuous.  Then $L(\G)$ is solid. 
\end{subthm} 

\begin{proof} Let $M \subset L(\G)$ be diffuse and $A \subset M$ be a masa (in $M$).  Since $A^\prime \cap L(\G) \supset M^\prime \cap L(\G)$ and there is a conditional expectation $A^\prime \cap L(\G) \to M^\prime \cap L(\G)$, it suffices to prove $A^\prime \cap L(\G)$ is injective. Since $M$ is diffuse, $A$ is non-atomic -- i.e., $A \cong L^\infty (\mathbb{T})$. Hence we can find a generating unitary $u \in A$ such that $u^n \to 0$ ultraweakly.  

Let $N = A^\prime \cap L(\G)$ and $\Phi\colon L(\G) \to N$ be the unique trace-preserving conditional expectation. As is well known, we can define a conditional expectation $\Psi\colon \B(\ell^2(\G)) \to A^\prime \cap \B(\ell^2(\G))$ by taking a cluster point of the maps $$\p_n(T) := \frac{1}{n} \sum_{i=1}^n u^i T u^{-i}.$$  Note that $\Psi$ contains $\K(\ell^2(\G))$ in its kernel, since $u^n \to 0$ ultraweakly.

Evidently $\Psi|_{L(\G)}$ is a trace-preserving conditional expectation of $L(\G)$ onto $N$; hence, by uniqueness, $\Psi|_{L(\G)} = \Phi$. Moreover, since $C^*_{\rho}(\G) \subset  A^\prime \cap \B(\ell^2(\G))$ we have that $C^*_{\rho}(\G)$ lies in the multiplicative domain of $\Psi$.  Thus, $$\Psi(\sum_{j=1}^k x_j y_j) = \sum_{j=1}^k \Phi(x_j) y_j,$$ for all $x_j \in L(\G)$ and $y_j \in C^*_{\rho}(\G)$.  In particular, $$\Phi|_{C^*_\lambda(\G)} \times \id_{C^*_{\rho}(\G)} (x \otimes y) = \Phi(x)y = \Psi(xy),$$ for all $x \in C^*_{\lambda} (\G)$ and $y \in C^*_{\rho}(\G)$.

By Lemma \ref{lem:thelem}, to prove $N$ is injective, it suffices to show that $$\Phi|_{C^*_\lambda(\G)} \times \id_{C^*_{\rho}(\G)} \colon C^*_\lambda(\G) \odot C^*_{\rho}(\G) \to \B(\ell^2(\G))$$ is $\min$-continuous. But we assumed that $$C^*_\lambda(\G) \odot C^*_{\rho}(\G) \to \B(\ell^2(\G)) \to \B(\ell^2(\G))/\K(\ell^2(\G))$$ is $\min$-continuous, and the map $\Psi$ factors through the Calkin algebra (since 
$\K(\ell^2(\G))$ is in its kernel), so $\min$-continuity follows from the fact that $\Phi|_{C^*_\lambda(\G)} \times \id_{C^*_{\rho}(\G)} (x \otimes y) = \Psi(xy)$ for all $x \in C^*_{\lambda} (\G)$ and $y \in C^*_{\rho}(\G)$.
\end{proof}

Recall that $\G$ is amenable if and only if the canonical $*$-homomorphism $C^*_\lambda(\G) \odot C^*_{\rho}(\G) \to \B(\ell^2(\G))$ is $\min$-continuous (Theorem \ref{thm:amenablegroup}).  Since the hypotheses of Ozawa's result looks frighteningly close to this, one should worry about the existence of nonamenable groups to which the result applies.  However, it turns out that many such groups exist. Even better, proving that they exist is a beautiful synthesis of geometric group theory, nuclearity and C$^*$-crossed products.  Hence our next section is a quick summary of the theory of amenable actions and their crossed products.  After that we'll come to examples. 

\section{Crossed Products, Amenable Actions and the Roe Algebra} 

The crossed product construction is fundamental in operator algebras and appears in several texts. Hence we won't recall the details -- see \cite[Section 4.1]{Me-n-Taka} -- but we will review the set-up. If $\G$ is a discrete group and $\alpha\colon \G \to \Aut(A)$ is a group homomorphism into the group of automorphisms of a C$^*$-algebra $A$, then one can construct a ``full" (or ``universal") crossed product $A \rtimes_\alpha \G$ -- satisfying the universal property that any covariant representation of $(A, \G, \alpha)$ extends to $A \rtimes_\alpha \G$ -- and a ``reduced" crossed product $A \rtimes_{\alpha,r} \G$ that is defined via a left-representation-like covariant representation.  When $\G$ is amenable, the two constructions yield the same algebra (see \cite[Theorem 4.6.2]{Me-n-Taka}). However, we'll be interested mainly in nonamenable groups -- that \emph{act} amenably.

\subsection{Amenable Actions} 

We've seen that amenable groups are defined via the canonical action of $\G$ on $\ell^\infty(\G)$. To define an amenable action we consider another canonical action.  

\begin{subdefn}  $\prob(\Gamma)$ is the set of probability measures on $\Gamma$ -- which
we identify with the set of positive, norm-one elements in $\ell^1(\Gamma)$ and topologize by restricting the weak-$*$ topology (coming from $c_0(\G)$). 
\end{subdefn} 

Note that $\G$ acts by left translation on $\ell^1(\Gamma)$ (just restrict the action on $\ell^\infty(\G)$ to absolutely summable sequences) and this leaves $\prob(\Gamma)$ invariant.  Hence we have a canonical action of $\G$ on $\prob(\Gamma)$ which we'll denote by $m \mapsto s.m$ for all $s \in \G$ and $m \in \prob(\Gamma)$. (As before, $s.m(g) = m(s^{-1}g)$.) Though $\prob(\Gamma)$ is never compact (when $\G$ is infinite), it is the prototype of an amenable action.

We call a compact Hausdorff space $X$ a \emph{$\G$-space} if it is equipped with an 
action of $\G$ (by homeomorphisms). 
We let $x \mapsto s.x$
denote the action of $s \in \G$ on $x \in X$. 

\begin{subdefn}\label{defn:amenableactionII} 
An action of $\G$ on a compact space
$X$ is called (topologically) {\em amenable} (or, equivalently, $X$
is an \emph{amenable $\G$-space}) if there exists a net of continuous
maps $m_i\colon X\to \prob(\Gamma)$, such that for each $s \in \Gamma$,
$$\lim_{i \to \infty}\Big( \sup_{x \in X} \| s.m_i^x -
m_i^{s.x}\|_1 \Big) = 0,$$ where $m_i^{x} \in \prob(\Gamma)$ is the probability measure that $m_i$ associates to the point $x \in X$.\footnote{Note that $m\colon
X\to \prob(\Gamma)$ is continuous if and only if for each convergent
net $x_i \to x \in X$ we have $m^{x_i}(g) \to m^x(g)$ for all $g \in
\Gamma$.}
\end{subdefn}

In other words, an action of $\G$ on $X$ is amenable if it can be asymptotically intertwined with the canonical action of $\G$ on $\prob(\Gamma)$: 
$$\begin{CD}
\hspace{5mm} X \hspace{5mm} @>m_i>> \hspace{5mm} \prob(\Gamma)\\ 
@V\mathrm{Given}V\mathrm{action}V @V\mathrm{Canonical}V\mathrm{action}V\\ 
\hspace{5mm} X \hspace{5mm}  @>m_i>> \hspace{5mm}  \prob(\Gamma).\\ 
\end{CD}$$

It is worthwhile to check that every action of an amenable group is amenable in the sense defined above. (If $F_i \subset \G$ is a F{\o}lner sequence, let $m_i$ map every point to the normalized characteristic function over $F_i$.) However, it is a remarkable fact that virtually all groups one normally encounters act amenably on some compact space. (For example, all linear groups -- see \cite{GHW}.) 

\begin{subrem} 
\label{rem:borel}
Definition \ref{defn:amenableactionII} requires the maps $m_i$ to be continuous. However, it is sufficient for the $m_i$'s to be Borel, meaning that for every $g \in \G$ the function $X \to \R$, $x \mapsto m_i^x(g)$ is Borel.  See \cite[Proposition 5.2.1]{Me-n-Taka} for a proof of this useful fact. 
\end{subrem} 

The following result connects exactness with amenable actions.  It is due to Ozawa and, independently, Anantharaman-Delaroche, following an important contribution by Guentner and Kaminker.  See \cite[Chapter 5]{Me-n-Taka} for details. 

\begin{subthm}
\label{thm:hmm}
For a discrete group $\G$, the following are equivalent:
\begin{enumerate}
\item\label{thm:actsc1}
$\G$ is exact;
\item\label{thm:actsc3}
$\G$ acts amenably on some compact topological space.
\end{enumerate}
\end{subthm}

For our purposes, the most important fact about amenable actions is that the associated crossed products are nuclear.  The proof is somewhat technical, but is close in spirit to the proof of $(\ref{thm:grpam3}) \Longrightarrow (\ref{thm:grpam9})$ in Theorem \ref{thm:amenablegroup}, in that one can construct approximating maps by hand.  See Theorem 4.3.4 and Lemma 4.3.7 in \cite{Me-n-Taka} for a proof of the following fact. 

\begin{subthm}\label{thm:crossedproductamenableaction}
If $\G$ acts amenably on $X$, then the associated (reduced) crossed product is nuclear.\footnote{It turns out that the full and reduced crossed products by amenable actions are isomorphic (\cite[Theorem 4.3.4]{Me-n-Taka}).}
\end{subthm}

There is much more that could, and probably should, be said about amenable actions, but I don't want to distract us from our goal of W$^*$-applications.  Hence I'll close this subsection with an exercise -- a very important exercise that we'll need later. 

\begin{subexer} 
\label{exer:biexact}  Assume $\G \times \G$ acts on $C(X)$ and there
exist two $\G\times \G$-invariant subalgebras $A,B \subset C(X)$
such that (i) $\G \times \{e\}|_A$ is amenable while $\{e\}\times \G|_A$
is trivial and (ii) $\G \times \{e\}|_B$ is trivial while
$\{e\}\times \G|_B$ is amenable.  Prove that the action of $\G \times \G$
on $C(X)$ is amenable. 
\end{subexer}

\subsection{The Uniform Roe Algebra} 

Since our goal is applications to group von Neumann algebras, it will be necessary to stay on the Hilbert space $\ell^2(\G)$. Crossed products don't naturally act on this Hilbert space; at least, it isn't obvious that they do.  

\begin{subdefn}  The \emph{uniform Roe algebra}  $C^*_u(\G)$ of $\G$ is the C$^*$-subalgebra of
$\B(\ell^2(\G))$ generated by $C^*_\lambda(\G)$ and
$\ell^\infty(\G)$.
\end{subdefn} 

It turns out that the uniform Roe algebra is a crossed product -- a very important one, too. The following result is a simple computation, but since we didn't define reduced crossed products, we refer the reader to \cite[Proposition 5.1.3]{Me-n-Taka} for a proof. 

\begin{subprop}
\label{prop:Roecross}
Let $\alpha\colon\Gamma \to \Aut(\ell^{\infty}(\Gamma))$ be the left translation action.  Then there is an isomorphism $$\pi\colon C^*_u(\G) \to \ell^{\infty}(\Gamma)\rtimes_{\alpha,r} \Gamma$$ such that $\pi(f) = f$ and $\pi(\lambda_s) = \lambda_s$ for all $f \in \ell^{\infty}(\Gamma)$ and $s \in \G$.
\end{subprop}

Note that $\ell^{\infty}(\Gamma)\rtimes_{\alpha,r} \Gamma$ is universal in the following sense: If $X$ is a compact $\G$-space, then there is a covariant homomorphism $C(X) \to \ell^{\infty}(\Gamma)$ and, thus, a $*$-homomorphism $C(X) \rtimes_r \G \to \ell^{\infty}(\Gamma)\rtimes_{\alpha,r} \Gamma$. To see this simply pick a point $x \in X$ and consider the orbit $\{ s.x : s \in \G\}$. The (automatically) continuous map $\G \to X$ defined by $s \mapsto s.x$ extends to a continuous map from the Stone-\v{C}ech compactification $\beta \G$ to $X$. Moreover, it is equivariant when $\beta \G$ is equipped with the (extension of the) left translation action of $\G$ on $\G$.  In other words, picking a point $x \in X$ determines a commutative diagram 
$$\begin{CD}
\beta\G @>>> X\\ 
@V\mathrm{Canonical}V\mathrm{action}V @V\mathrm{\G}V\mathrm{action}V\\ 
\beta \G   @>>> X.\\ 
\end{CD}$$

From this observation it is easy to see that if $\G$ acts amenably on $X$, then the left translation action of $\G$ on $\ell^{\infty}(\Gamma)$ is also amenable. It follows that $\G$ admits an amenable action on some compact space $X$ if and only if the left translation action of $\G$ on $\ell^{\infty}(\Gamma)$ is  amenable. This remark is essentially useless if one hopes to determine whether or not a given group acts amenably on some space.  But it's useful for other things; combining the results above, we have 

\begin{subcor}  $\Gamma$ admits an amenable action on some compact Hausdorff space if and only if  $C^*_u(\G)$ is nuclear. 
\end{subcor}

\subsection{Small Compactifications and the Akemann-Ostrand property} 

The problem of constructing amenable actions is a hard one, at least for an operator algebraist.  We'll soon see that our friends in geometric group theory are far better equipped to handle this problem, but before getting concrete let's set the stage with some general preparations.

\begin{subdefn}
\label{defn:grpcompactification}
A \emph{(covariant) compactification} of a group $\G$ is a compact topological space
$\bar{\G}=\G\cup\partial\G$ containing $\G$ as an open dense subset and with the property that the left translation action of $\G$ on $\G$
extends to a continuous action of $\G$ on $\bar{\G}$.
\end{subdefn}

For such a compact space $\bar{\G}$, there is a canonical embedding $C(\bar{\G}) \rtimes_r \G \subset C^*_u(\G)$ determined by the neutral element $e \in\G$.  Namely, a function $f \in C(\bar{\G})$ is identified with $(f(g))_{g\in \G} \in \ell^{\infty} (\G)$ while $\lambda_s \in C(\bar{\G}) \rtimes_r \G$ is identified with $\lambda_s \in C^*_u(\G)$. (Injectivity follows from the density of $\G$ in $\bar{\G}$.) \emph{For the remainder of these notes, we will make this identification without reference.} 

\begin{subdefn}
\label{defn:small}
A compactification $\bar{\G}$ is said to be \emph{small at infinity}
if for every net $\{ s_n\} \subset \G$ converging to a boundary point $x\in\partial\G$
and every $t\in\G$, one has that $s_nt\to x$.\footnote{For those in the know, this is the same as saying that $\bar{\G}$ is a quotient of the Higson corona of $\G$.}
\end{subdefn}

To understand this definition, recall that $\G$ also acts on $\ell^{\infty} (\G)$ by right translation; i.e., the mapping $f \mapsto f^t$, where $f \in\ell^{\infty} (\G)$, $t \in \G$ and $f^t(g) = f(gt^{-1})$ for all $g \in \G$, defines a right action of $\G$ on $\ell^{\infty} (\G)$.  It is worth pointing out that this action, like its leftist counterpart, is spatially implemented. But this time the right regular representation does the job: $f^t = \rho_t^* f \rho_t$, for all $f \in\ell^{\infty} (\G)$ and $t \in \G$. 

To see what being small at infinity has to do with the right translation action, we suggest proving the following lemma. 

\begin{sublem}
Let $\G$ be a group and $\bar{\G}=\G\cup\partial\G$
be a compactification.
The following are equivalent:
\begin{enumerate}
\item
the compactification $\bar{\G}$ is small at infinity;
\item
the right translation action extends to a continuous action
on $\bar{\G}$ in such a way that it is trivial on $\partial\G$;
\item
one has $f^t-f\in c_0(\G)$ for every $f\in C(\bar{\G})$ and $t\in\G$.
\end{enumerate}
\end{sublem}

Having wandered into the trees a bit, let's step back and take a look at the forest. Recall that our goal is to give examples of groups satisfying the hypotheses of Theorem \ref{thm:taka}. Our next result says that it suffices to find groups with sufficiently nice compactifications.

\begin{subprop}
\label{prop:AO}
Assume $\G$ admits a compactification $\bar{\G}$ which is small at infinity and such that the left action of $\G$ on $\bar{\G}$ is amenable. 
Then the canonical map $$C^*_\lambda(\G) \odot C^*_{\rho}(\G) \to \B(\ell^2(\G)) \to \B(\ell^2(\G))/\K(\ell^2(\G))$$is $\min$-continuous. Hence, $L(\G)$ is solid.\footnote{Since every subalgebra of a nuclear C$^*$-algebra is exact, the existence of an amenable action implies exactness of the group (since the associated crossed product is nuclear and contains $C^*_\lambda(\G)$; cf.\ Theorem \ref{thm:hmm}). Thus the hypotheses of this proposition imply that Theorem \ref{thm:taka} is applicable.} 
 
\end{subprop}
\begin{proof}
It suffices to show that there exists a nuclear C$^*$-algebra $A \subset \B(\ell^2(\G))$
such that $C^*_\lambda(\G) \subset A$ and $\pi(A)$ commutes with $\pi(C^*_\rho(\G))$, where $\pi\colon \B(\ell^2(\G)) \to \B(\ell^2(\G))/\K(\ell^2(\G))$ is the quotient map.
Indeed, if such $A$ exists, then we have an inclusion
$C^*_\lambda(\G)\otimes C^*_\rho(\G)\subset A \otimes C^*_\rho(\G) = A \mtp C^*_\rho(\G)$
and a natural $*$-homomorphism $A \mtp C^*_\rho(\G) \to \B(\ell^2(\G))/\K(\ell^2(\G))$.

So, let $\bar{\G}$ be a left-amenable compactification which is small at infinity.  According to Theorem \ref{thm:crossedproductamenableaction}, the crossed product $$A := C(\bar{\G}) \rtimes_r \G \subset C^*_u(\G) \subset \B(\ell^2(\G))$$ is nuclear and contains $C^*_\lambda(\G)$.  So we only have to check that $A$ commutes with $C^*_\rho(\G)$, after passing to the Calkin algebra.  

Since $A$ is generated by $C(\bar{\G})$ and $C^*_\lambda(\G)$ -- and $C^*_\rho(\G)$ commutes with $C^*_\lambda(\G)$ -- it suffices to check that $[f,\rho_t] \in \K(\ell^2(\G))$ for every $f \in C(\bar{\G})$ and $t \in \G$.  Which is where being small at infinity comes in.  Indeed, condition (3) in the previous lemma says that 
\[
\rho_t^*f\rho_t-f=f^t-f\in c_0(\G)\subset\K(\ell^2(\G))
\]
for any $f\in C(\bar{\G})$ and any $t\in\G$. Thus $[f,\rho_t] = \rho_t(\rho_t^*f\rho_t-f) \in \K(\ell^2(\G))$, as desired. 
\end{proof}

\begin{subrem}  Akemann and Ostrand first proved that if $\G$ is a free group, then $$C^*_\lambda(\G) \odot C^*_{\rho}(\G) \to \B(\ell^2(\G)) \to \B(\ell^2(\G))/\K(\ell^2(\G))$$ is $\min$-continuous (cf.\ \cite{akemann-ostrand}). Hence, we refer to this as the \emph{Akemann-Ostrand property}. 
\end{subrem} 

A theorem of Choi implies that free groups are exact (cf.\ \cite{choi:free}). Together with the work of Akemann and Ostrand, this implies that free groups satisfy the hypotheses of Theorem \ref{thm:taka}.  Hence, modulo citation, we have shown that free group factors are solid.\footnote{Popa has found a W$^*$-proof of solidity of free group factors (cf.\ \cite{popa:solid}), but it doesn't extend to the hyperbolic groups and other examples that C$^*$-techniques can reach.}  In particular, we deduce a celebrated result that was first proved using free probability theory. 

\begin{subthm}[Ge \cite{ge}] Free group factors are prime.  
\end{subthm}

Though there is nothing wrong with relying on the work of Choi and Akemann-Ostrand, it turns out that more can be said -- much more -- if we instead prove that free groups satisfy the hypotheses of Proposition \ref{prop:AO}.  That will be the topic of the next section.

\section{The Free Group $\F_2$} 

In this section we exploit Gromov's view of discrete groups as geometric objects to show that the free group on two generators satisfies the hypotheses of Proposition \ref{prop:AO}. It turns out that this approach works equally well for general hyperbolic groups, but the ideas are particularly transparent if we first look at $\F_2$.  

\subsection{A Compactification of $\F_2$}

Recall that the Cayley graph of a group has vertices labeled by the group elements and an edge between two vertices whenever the corresponding group elements differ by a generator (in particular, it depends on the choice of a generating set).  If we choose the canonical generating set of $\F_2$, denoted by the letters $\{a, b, a^{-1}, b^{-1}\}$, then the ``center" of the Cayley graph looks like this:\footnote{Sorry, but I can't do pictures in LaTeX. I can't even do them by hand, so Yuri saved me (again).}



\vspace{3mm}

\hspace{2.5cm}\includegraphics[scale=0.75]{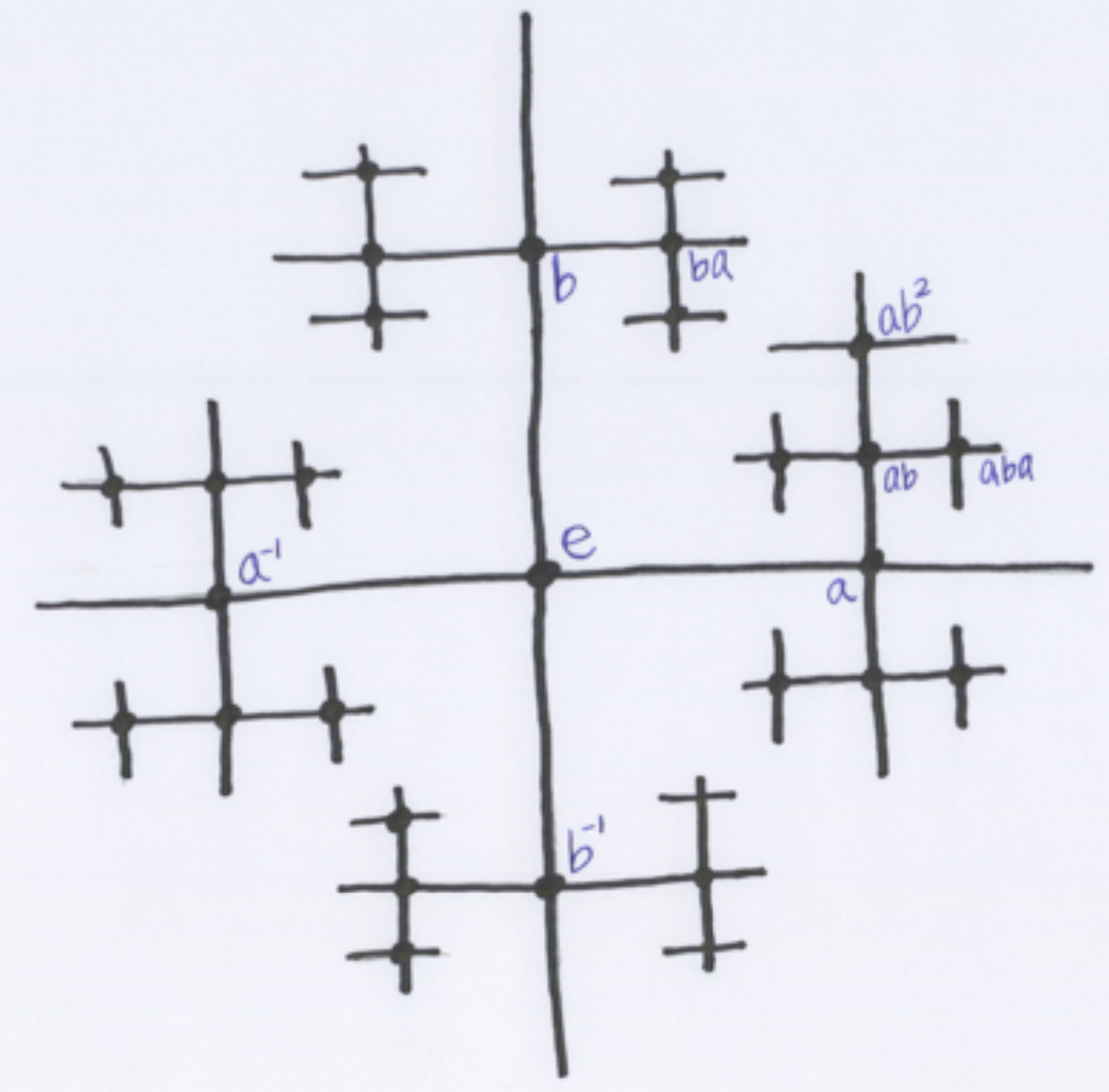}

To construct the right compactification of $\F_2$, we have to consider geodesics in the Cayley graph -- i.e., paths that never backtrack.  Note that since the Cayley graph of $\F_2$ is a tree, there is a \emph{unique} geodesic connecting any pair of vertices.  However, we must consider infinite geodesics, too. 

\begin{subdefn} Let $\bar{\F}_2$ denote the set of geodesics (both finite and infinite) in the Cayley graph of $\F_2$ which start at the neutral element $e$. 
\end{subdefn} 

We can view this compactification as follows:  

\vspace{3mm}

\hspace{2cm}\includegraphics[scale=0.60]{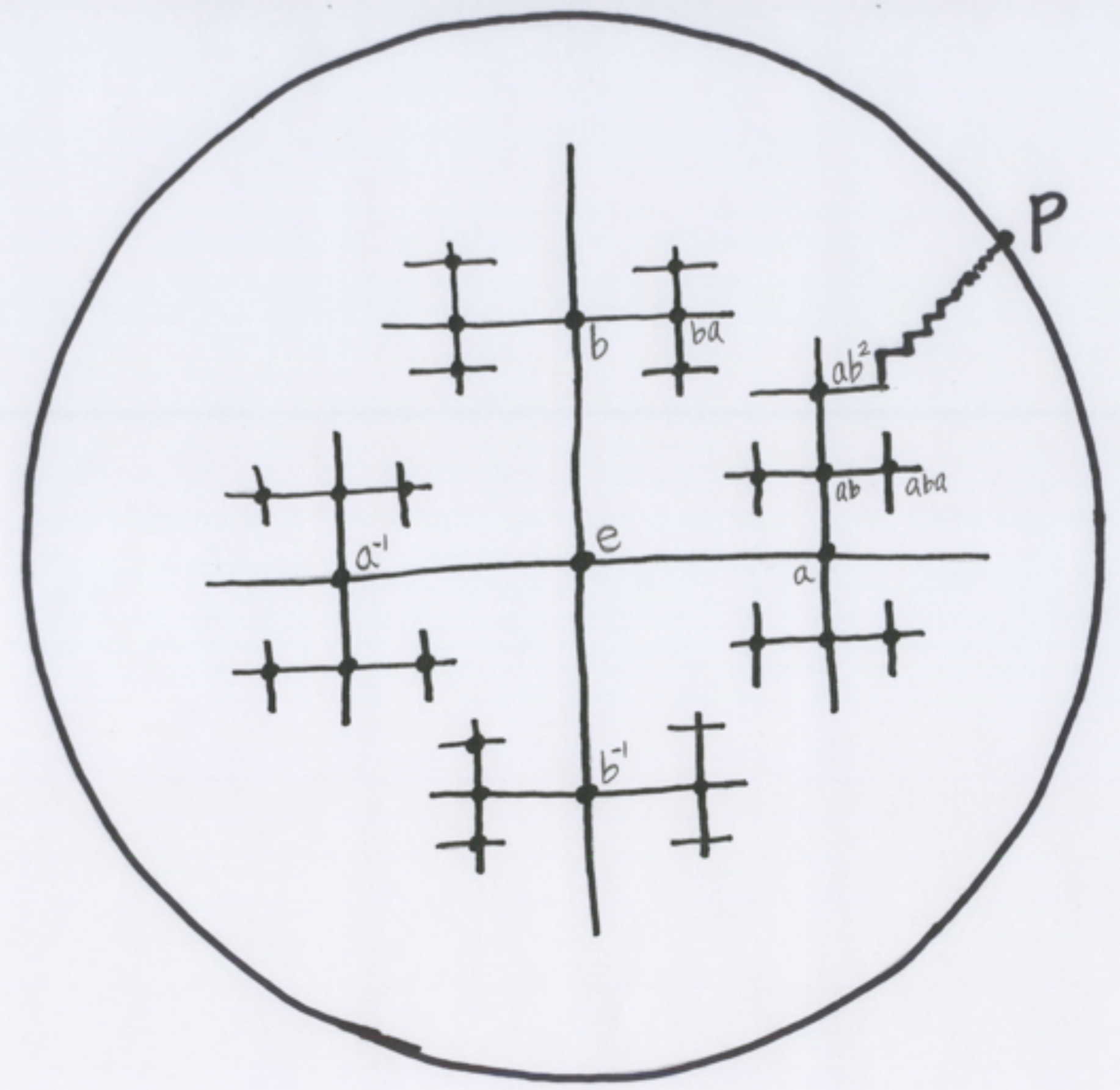}

{\noindent The} circle around the outside represents the infinite geodesics and we think of each such path as converging to a boundary point.

Let me first describe the topology on $\bar{\F}_2$ in words.  If $p \in \bar{\F}_2$ is a finite path, then a net of paths $\{ p_\lambda \}$ converges to $p$ if and only if $p_\lambda = p$ for all large $\lambda$. However, if $p$ is infinite, then $p_\lambda \to p$ means that the $p_\lambda$'s agree with $p$ more and more and more; a bit more precisely, for every $k \in \N$ there exists a $\lambda_0$ such that the first $k$-steps along the path $p_\lambda$ agree with the first $k$-steps along $p$, for all $\lambda \geq \lambda_0$.  However, the topology on $\bar{\F}_2$ is best understood pictorially; $p_n \to p$ means: 

\vspace{3mm}

\hspace{3.5cm}\includegraphics[scale=0.50]{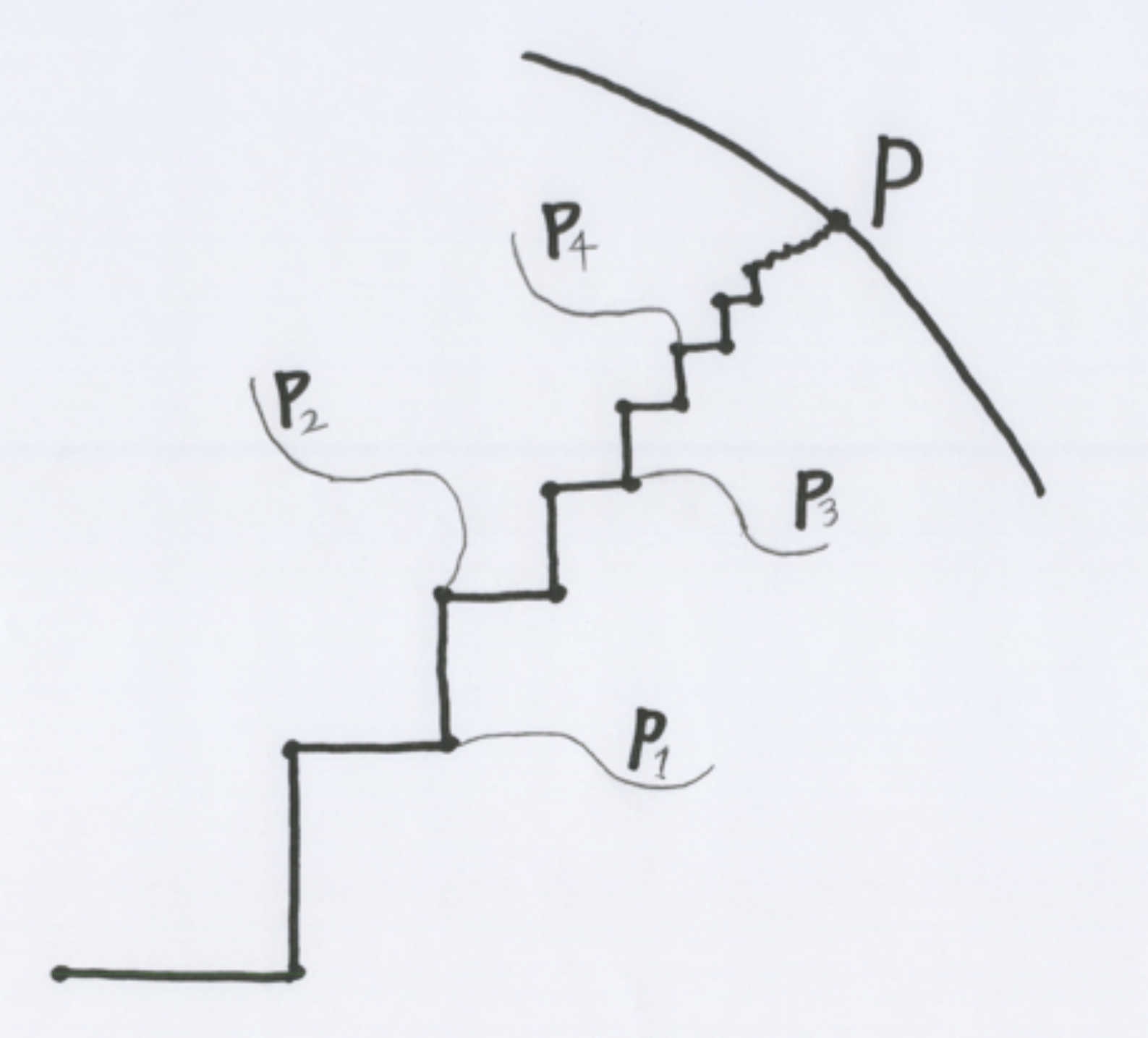}


For those uncomfortable with loose definitions (like me), let's be more rigorous.  First we identify a path $p \in \bar{\F}_2$ with a (finite or infinite) reduced word in the obvious way: 

\vspace{3mm}

\hspace{2.5cm}\includegraphics[scale=0.55]{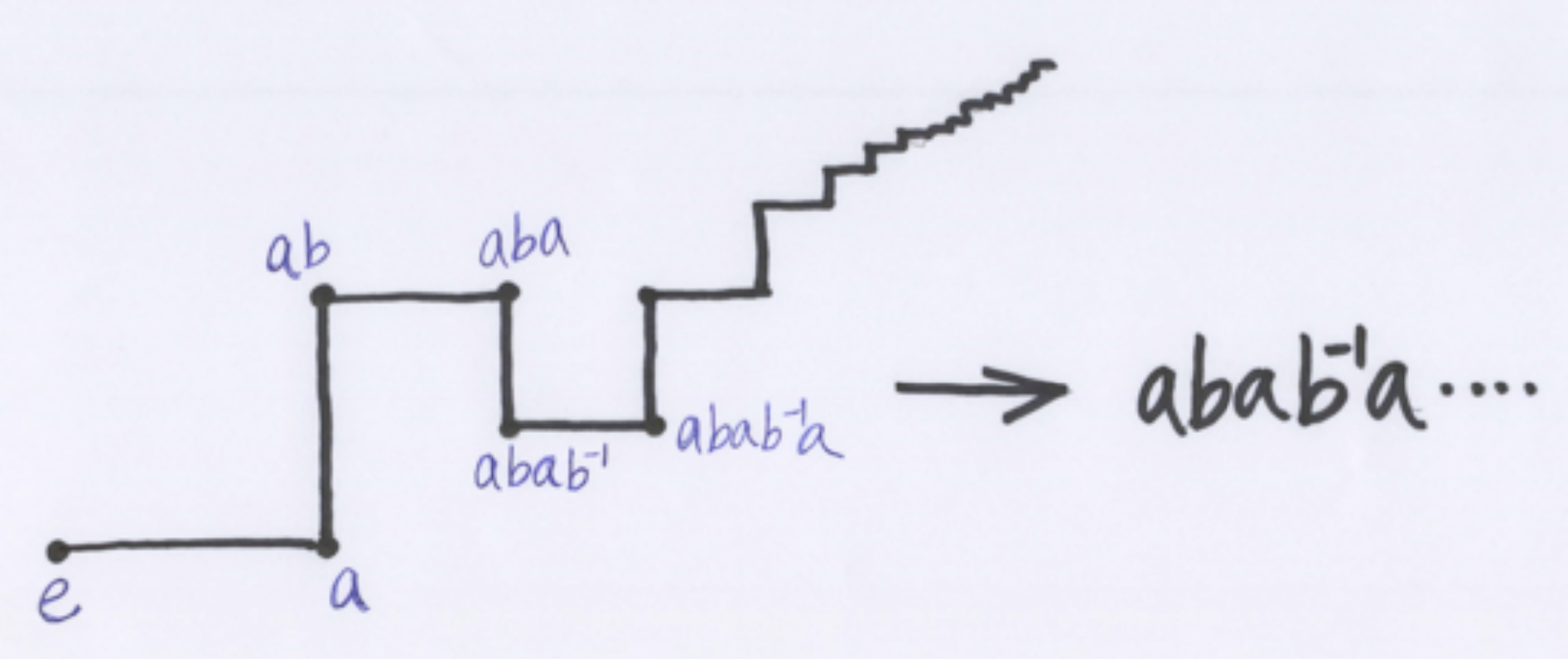}

Then we identify reduced words with the sequences $(x_i)$ in $\prod_\N \{a, b, a^{-1}, b^{-1}\}$ determined by the following rule: Either $x_{i+1} \neq x_i^{-1}$ for all $i \in \N$ (infinite words) or, if $x_{i+1} = x_i^{-1}$, then $x_j = x_i^{-1}$ for all $j > i$ (finite words).  It is a good exercise to check that restricting the product topology on $\prod_\N \{a, b, a^{-1}, b^{-1}\}$ to this copy of  $\bar{\F}_2$ yields the topology described above. Moreover, the complement of $\bar{\F}_2$ in $\prod_\N \{a, b, a^{-1}, b^{-1}\}$ is open (another simple exercise), so $\bar{\F}_2$ is compact.  

\subsection{$\bar{\F}_2$ is an Amenable $\F_2$-space and Small at Infinity} 

Now we must show that $\bar{\F}_2$ is a compactification in the sense of Definition \ref{defn:grpcompactification}; that it is small at infinity; and that the left action of $\F_2$ on $\bar{\F}_2$ is amenable. 

First, we identify $\F_2$ with the set of finite paths in $\bar{\F}_2$. It is easy to check that $\F_2 \subset \bar{\F}_2$ is dense and open.  To describe the left action of $\F_2$ on $\bar{\F}_2$, we use the reduced-word picture.  Indeed, if $s = x_1 x_2 \cdots x_k \in \F_2$ is a reduced word and $y = y_1y_2 \cdots \in \bar{\F}_2$ then $s$ acts on $y$ by left concatenation and cancellation of any inverses, i.e., if $x_k \neq y_1^{-1}$, then $s.y = x_1 x_2 \cdots x_k y_1y_2 \cdots \in \bar{\F}_2$ (and otherwise a bit of cancellation may occur).  Clearly this extends the left action of $\F_2$ on itself, and it's easily seen to be continuous.  Thus $\bar{\F}_2$ is a compactification. 

Next we check the small-at-infinity condition -- which is my favorite part, because pictorially it's completely trivial.  Indeed, the right action of $\F_2$ on itself is also by concatenation so if $p_n \in \F_2 \subset \bar{\F}_2$ converge to an infinite path $p \in \bar{\F}_2$ and $t \in \F_2$ is arbitrary, then evidently $p_n t \to p$ as  well. 

\vspace{3mm}

\hspace{4cm}\includegraphics[scale=0.45]{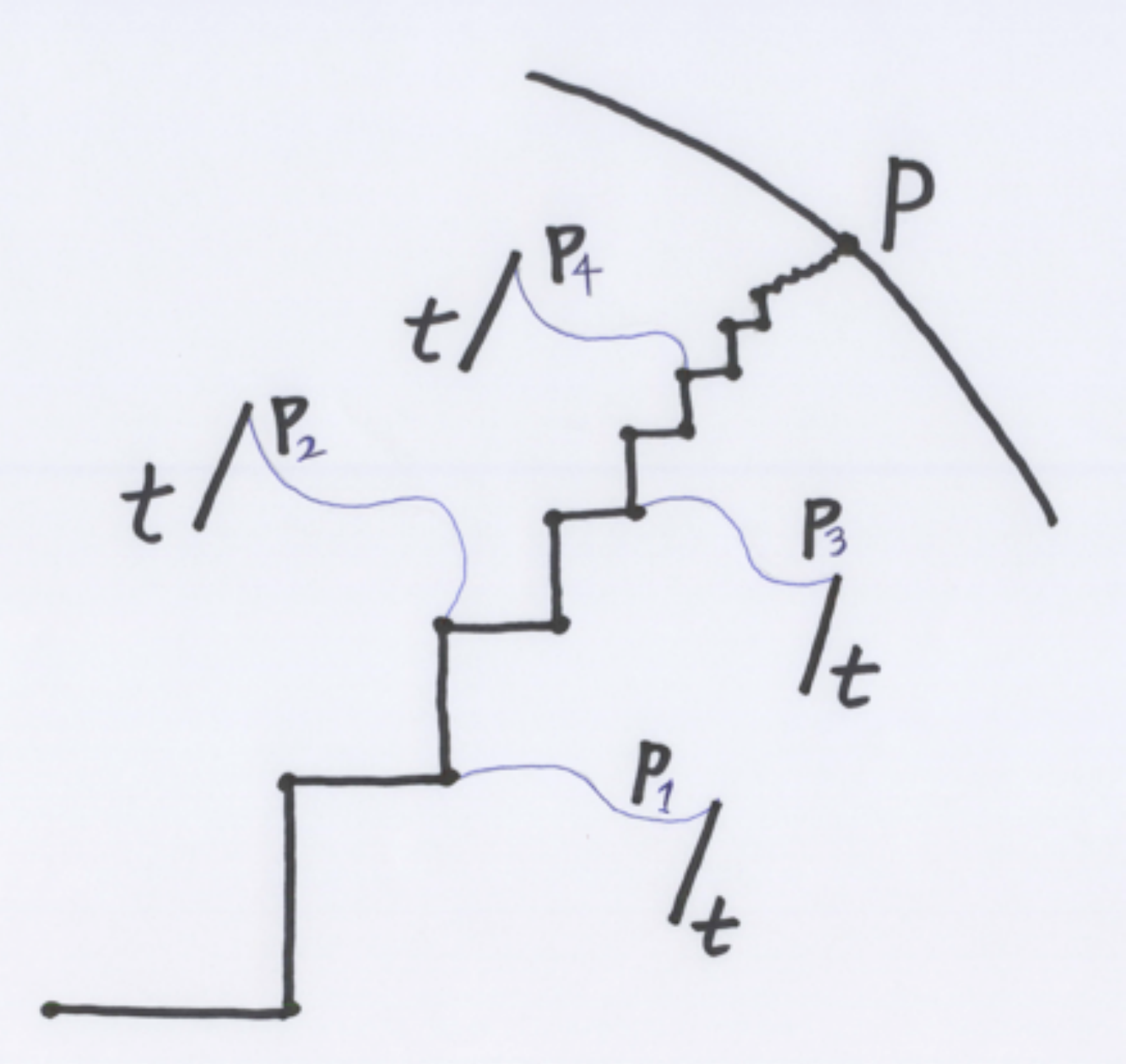}

Finally, we must see why the left action of $\F_2$ on $\bar{\F}_2$ is amenable.  That is, we must construct continuous maps $m_i\colon \bar{\F}_2 \to \prob(\F_2)$, such that for each $s \in \F_2$, $$\lim_{i \to \infty}\Big( \sup_{x \in \bar{\F}_2} \| s.m_i^x - m_i^{s.x}\|_1 \Big) = 0.$$ Given $x \in \bar{\F}_2$, we write it as a reduced word $x = x_1x_2x_3\cdots$ and define $x(k) = x_1x_2\cdots x_k$, where $x(0) := e$.  (If $x = x_1\cdots x_j$ is a finite word, then we put $x(k) = x$ for all $k \geq j$.)  Now, for $N \in \N$ we define $m_N \colon \bar{\F}_2 \to \prob(\F_2)$ by $$m_N^x = \frac{1}{N} \sum_{k = 0}^{N-1} \delta_{x(k)},$$ where $\delta_{x(k)}$ is the point mass concentrated at $x(k)$. In the case that $x$ is an infinite word, $m_N^x$ is just the normalized characteristic function over the first $N$ steps along the path determined by $x$; when $x$ is a finite word, $m_N^x$ is converging to the point mass at $x$ (as $N \to \infty$). One checks that each $m_N$ is continuous, and for every $s \in \F_2$ we have $$\sup_{x \in \bar{\F}_2} \| s.m_N^x - m_N^{s.x}\|_1 \leq \frac{2|s|}{N},$$ where $|s|$ denotes the word length of $s$.  One can verify this inequality by hand, but it's easier to see geometrically.  Indeed, if $x = x_1x_2\cdots$ is an infinite word and $s = s_1\cdots s_m \in \F_2$, then after cancellation we have $s.x = s_1\cdots s_{d}x_{(m-d)+1} x_{(m-d)+2}\cdots$ for some $d\leq m$.  A computation shows that $s.m_N^x$ is just the normalized characteristic function over the first $N$ steps along the geodesic that starts at $s$, goes back to $s_1\cdots s_{d}$, and then proceeds along the remainder of $s.x$.  Thus, in pictures, we have: 

\vspace{3mm}

\hspace{2cm}\includegraphics[scale=0.75]{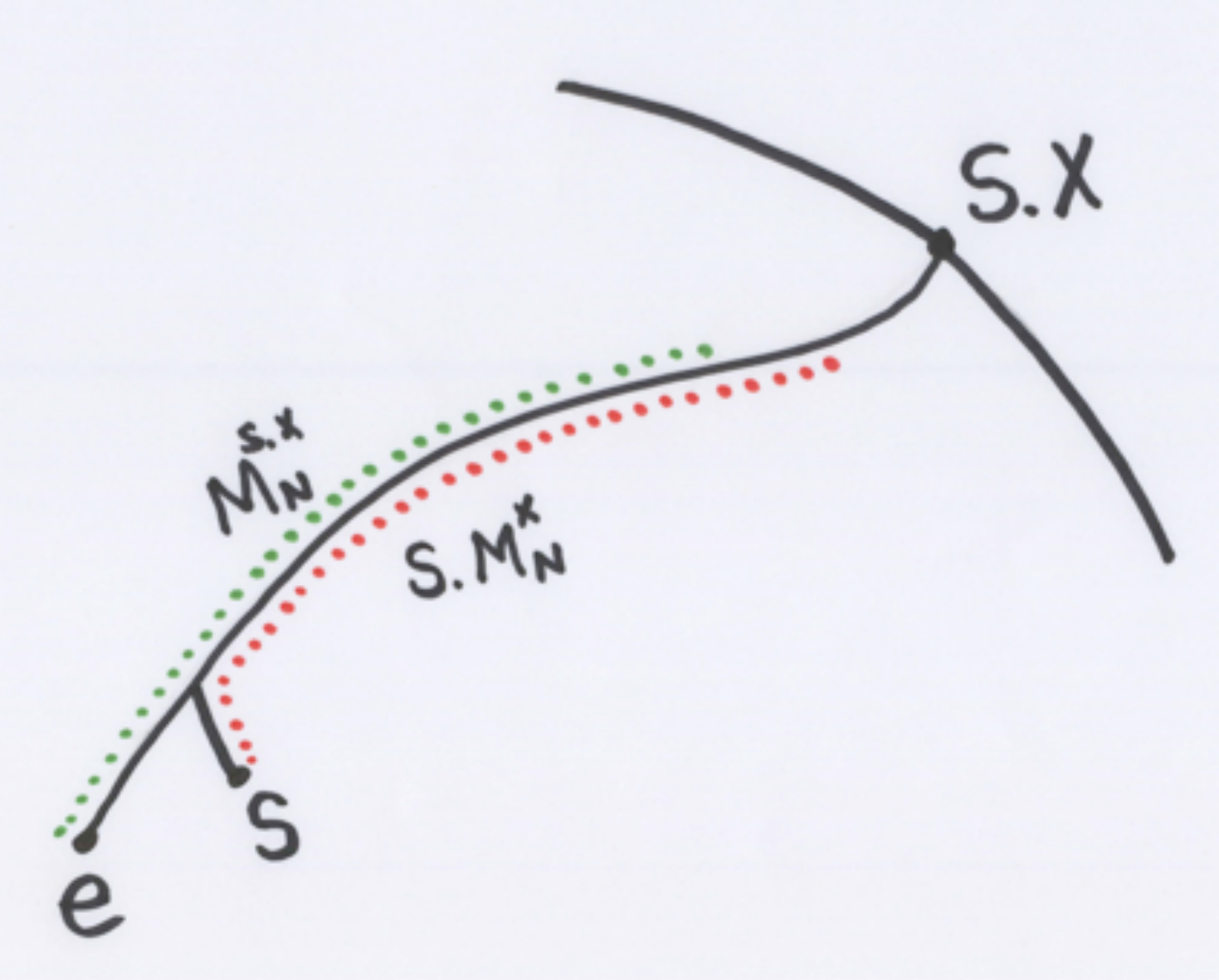}   

{\noindent It is} clear from this picture that $s.m_i^x - m_i^{s.x}$ is zero over the intersection of the two paths,  and the number of points where the paths don't overlap is bounded above by $2|s|$.  The case that $x$ is a finite path is similar, so we conclude that $\sup_{x \in \bar{\F}_2} \| s.m_N^x - m_N^{s.x}\|_1 \leq \frac{2|s|}{N}$. Letting $N \to \infty$ completes the proof.   

\subsection{Hyperbolic Groups} 

It turns out that everything we've done in this section extends to hyperbolic groups.  The details are significantly more annoying, but the geometric intuition is virtually identical.  Hence, I'll sketch the ideas, but refer the reader to \cite[Section 5.3]{Me-n-Taka} for details.

Let $K$ be a connected graph with the graph metric $d$ (the distance between vertices is the length of the shortest path connecting them). For every pair of vertices $x,y\in K$,
there exists a (not necessarily unique) geodesic path connecting $x$ to $y$ and we let $[x,y]$ denote such a geodesic (though more than one may exist).
For every subset $A\subset K$ and $r>0$, we define
\[
d(x,A)=\inf\{ d(x,a) : a\in A\}\ \mbox{ and }\
N_r(A)=\{ x\in K : d(x,A)< r\}.
\]
The set $N_r(A)$ is called the $r$-tubular neighborhood of $A$ in $K$.

%

\begin{subdefn}
Let $K$ be a connected graph. A  \emph{geodesic
triangle} $\triangle$ in $K$  consists of three points $x,y,z$ in $K$
and three geodesic paths $[x,y],[y,z],[z,x]$
connecting them.
\end{subdefn}

Quick question: If $K$ is a tree, what does a geodesic triangle look like?  That's right, it's awfully skinny. 

\begin{subdefn}[Hyperbolic graph]\label{defn:hyperbolicgraph}
For $\delta>0$, we say a geodesic triangle
$\triangle$ is \emph{$\delta$-slim} if
each of its sides is contained in the open
$\delta$-tubular neighborhood of the union of the other two --
i.e., $[x,y]\subset N_\delta([y,z]\cup[z,x])$
and similarly for the other two sides.
We say that the graph $K$ is \emph{hyperbolic}
if there exists $\delta>0$ such that
every geodesic triangle in $K$ is $\delta$-slim.
\end{subdefn}

\hspace{2cm}\includegraphics[scale=0.75]{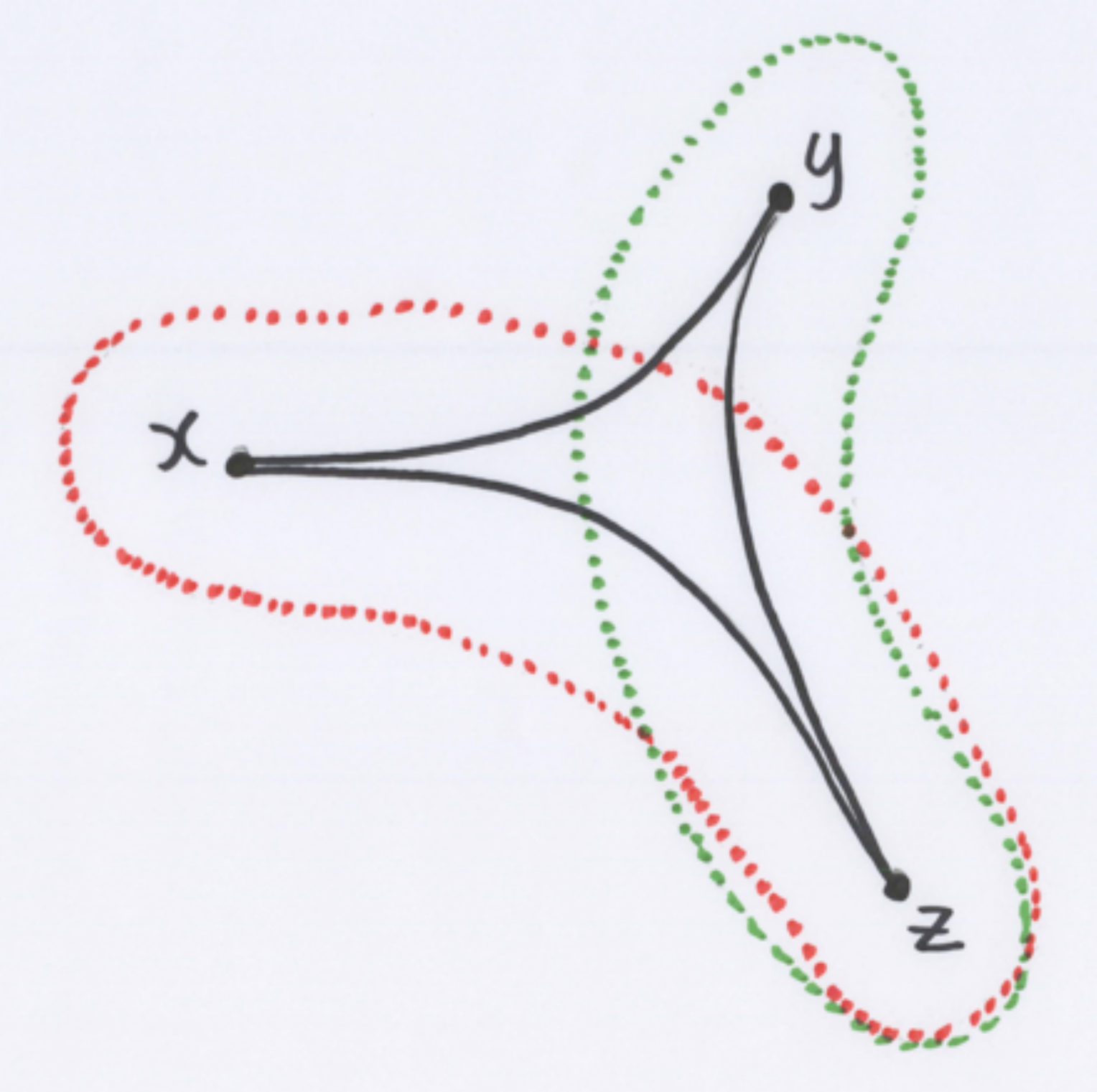}
%


\begin{subdefn}[Hyperbolic group]
Let $\G$ be a finitely generated group.
We say that $\G$ is \emph{hyperbolic} if its Cayley graph is hyperbolic.
\end{subdefn}

It can be shown that this definition does not depend on the choice of generating set -- i.e., it's an intrinsic property of the group $\G$. 

\begin{subrem}  Clearly all finite groups are hyperbolic (take $\delta > |\G|$). Another simple example is $\Z$.  The simplest example of a non-hyperbolic group is $\Z^2$ (since, for any $\delta > 0$, one can find a geodesic triangle that is much fatter than $\delta$). 

Among nonamenable groups, free groups
are hyperbolic (since their Cayley graph is a tree).  Other examples include co-compact lattices in
simple Lie groups of real rank one and the fundamental groups of
compact Riemannian manifolds of negative sectional curvature (cf.\ \cite{gromov}).
\end{subrem}

If $\G$ is hyperbolic, then one can compactify it in much the same way as we did for $\F_2$.  Very roughly, the \emph{Gromov compactification} $\bar{\G} = \G \cup \partial \G$ is comprised of (certain equivalence classes of) geodesics in the Cayley graph, with $\partial \G$ corresponding to infinite paths. 
Geometrically, we again view $\bar{\G}$ as the Cayley graph of $\G$ with a large circle around it -- each point on the circle representing the (equivalence class of) infinite path(s) that ``points in that direction."  

Since the Cayley graph of $\G$ need not be a tree, the geometry of $\bar{\G}$ is not quite as nice as the free group case.  However, the $\delta$-slim condition forces really big geodesic triangles to look like the figure after Definition \ref{defn:hyperbolicgraph}. Hence, if two infinite geodesics head toward the same boundary point -- i.e., go off to infinity ``in the same direction" -- it looks something like our next picture.  

\hspace{2cm}\includegraphics[scale=0.75]{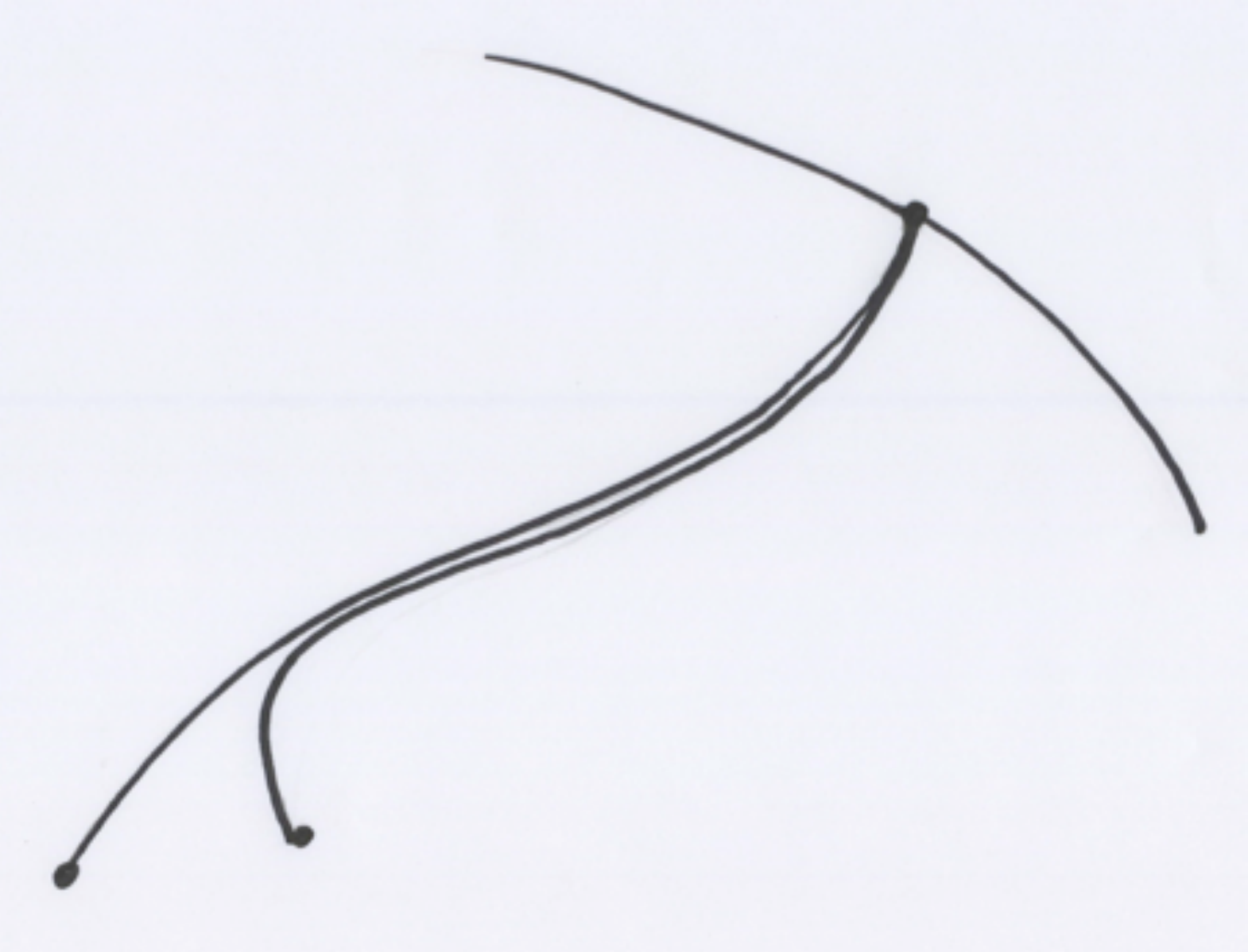}

To show that a hyperbolic group acts amenably on its Gromov boundary, one must construct maps $m_N \colon \bar{\G} \to \prob{\G}$ with the right properties.  For free groups, we used normalized characteristic functions concentrated on geodesics. But this only works because the Cayley graph is a tree and hence geodesics which point in the same direction must eventually flow together.  Thus, for general hyperbolic groups, we have to fatten up our characteristic functions a bit; namely, we take normalized characteristic functions over tubular neighborhoods of geodesics, as in our next illustration.  

\vspace{3mm}

\hspace{2cm}\includegraphics[scale=0.75]{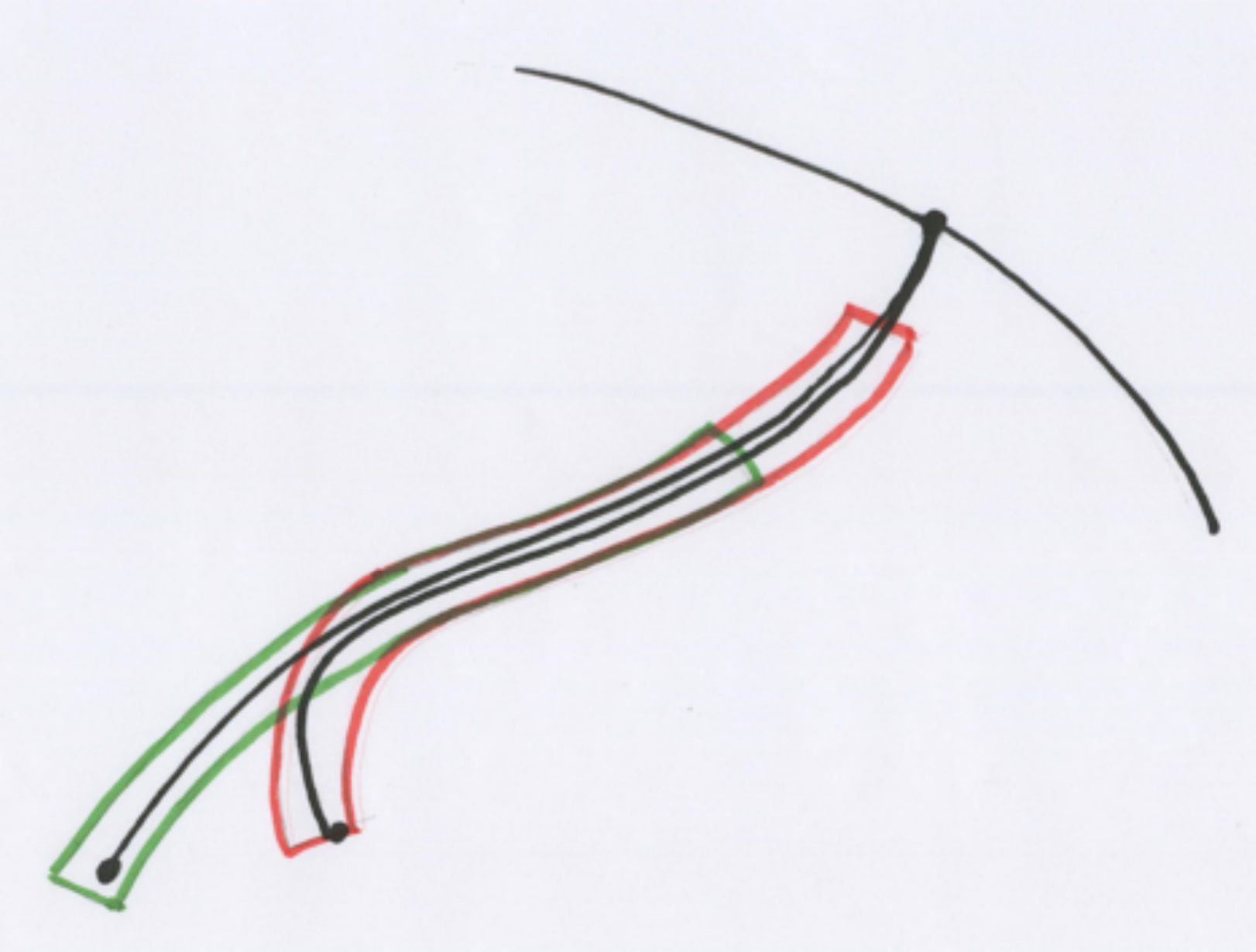}

Vigorously waving my hands, shouting down righteous objections and embracing the proof-by-intimidation mentality, I've essentially demonstrated the following theorem (see \cite[Theorem 5.3.15 and Proposition 5.3.18]{Me-n-Taka}).\footnote{By the way, this modus operandi is adored by Fundamentalists of all stripes -- please forgive the relapse. } 

\begin{subthm}\label{thm:grbdam} If $\G$ is hyperbolic, then it acts amenably (from the left) on $\bar{\G}$, and $\bar{\G}$ is small at infinity.  Hence, $L(\G)$ is solid. 
\end{subthm}

\section{More W$^*$-Applications} 

In addition to the hyperbolic-group generalization, there is another reason I've discussed the geometric-group-theory proof of exactness of free groups.  Namely, this approach leads to a concept which proves useful for further applications. 

We've seen that the left-translation action of a hyperbolic group $\G$ on $\ell^\infty(\G)$ is amenable,
but much more is true: The action of $\G\times \G$ on $\ell^\infty(\G)$
(given by the left and right translations) is amenable mod $c_0(\G)$.  

\begin{thm}  
\label{cor:hypbiexact}
If $\G$ is hyperbolic, then $\G\times \G$ acts amenably on $\ell^\infty(\G)/c_0(\G)$.
\end{thm} 

\begin{proof}  Since the Gromov boundary is small at infinity, identifying $C(\bar{\G}) \subset \ell^\infty(\G)$ and passing to the quotient $\ell^\infty(\G)/c_0(\G)$, we can find a $(\G\times \G)$-invariant
subalgebra $A \subset \ell^\infty(\G)/c_0(\G)$ with the following properties: The action of $\G\times \{e\}$ on $A$ is amenable, while the action of $\{e\}\times \G$ on $A$ is trivial. By symmetry, we can also find a $\G \times \G$-invariant subalgebra $B \subset \ell^\infty(\G)/c_0(\G)$ such that $\G\times \{e\}$ acts trivially on $B$, while $\{e\}\times \G$ acts amenably. 

The result now follows from Exercise \ref{exer:biexact}. 
\end{proof}

This is a special case of so-called \emph{bi-exactness}. 

\subsection{Bi-exactness} 

Let $\G$ be a group and $\grp$ be a family of subgroups of $\G$.  For a net $(s_i)$ in $\G$,  we write $s_i\to\infty/\grp$ if $s_i\notin s\Lambda t$ eventually\footnote{$\forall s,t, \forall\Lambda, \exists i_0$ such that $\forall i$ we have the implication $i\geq i_0\Rightarrow s_i\notin s\Lambda t$.} 
for every $s,t\in\G$ and $\Lambda\in\grp$. In other words, $s_i\to\infty/\grp$ if $s_i$ eventually escapes any finite set of translates of subgroups in $\grp$. Now we define $$c_0(\G;\grp) = \{ f \in \ell^\infty(\G): \lim_{s\to\infty/\grp}f(s)=0 \},$$ where the notation $\lim_{s\to\infty/\grp}f(s)=0$ means that for every net $(s_i)$ such that $s_i\to\infty/\grp$ and every $\e > 0$, there exists $i_0$ such that $|f(s_i)| < \e$ for all $i \geq i_0$.  

\begin{subexample}
Consider the case that $\grp$ consists only of the trivial subgroup $\{e\}$.  Then, $s_i\to\infty/\grp$ means that $s_i$ eventually escapes any finite set of elements. Thus one easily checks that $c_0(\G;\grp) = c_0(\G)$.  
\end{subexample}

Note that $c_0(\G;\grp)$ is an ideal in $\ell^\infty(\G)$ and, moreover, it is invariant under the left$\times$right-translation action of $\G \times \G$ on $\ell^\infty(\G)$; hence the left$\times$right action of $\G \times \G$ descends to the quotient algebra $\ell^\infty(\G)/c_0(\G;\grp)$.

\begin{subdefn}  We say $\G$ is \emph{bi-exact} relative to a family of subgroups $\grp$ if the left$\times$right action of $\G \times \G$ on $\ell^\infty(\G)/c_0(\G;\grp)$ is amenable.\footnote{This is not the definition given in \cite{Me-n-Taka}, but it's equivalent (see \cite[Proposition 15.2.3]{Me-n-Taka}).}
\end{subdefn} 

With this definition, Theorem \ref{cor:hypbiexact} says that hyperbolic groups are bi-exact relative to the trivial subgroup. 

Let $\K(\G;\grp)$ be the hereditary C$^*$-subalgebra of $\B(\ell^2(\G))$ 
generated by $c_0(\G;\grp)$: 
\[
\K(\G;\grp)=\mbox{the norm closure of }c_0(\G;\grp)\B(\ell^2(\G))c_0(\G;\grp).
\]
Since the left and right regular representations $\lambda$ and, respectively, $\rho$ 
normalize $c_0(\G;\grp)$, the reduced group C$^*$-algebras 
$C^*_\lambda(\G)$ and $C^*_\rho(\G)$ are in the multipliers of $\K(\G;\grp)$. 

Here's an analogue of the Akemann-Ostrand property for bi-exact groups. 

\begin{sublem}\label{lem:clrmid}
If $\G$ is bi-exact relative to $\grp$, then  
there exists a u.c.p.\ map 
\[
\theta\colon C^*_\lambda(\G)\otimes C^*_\rho(\G)\to\B(\ell^2(\G))
\]
such that $\theta(a\otimes b)-ab\in\K(\G;\grp)$ 
for every $a\in C^*_\lambda(\G)$ and $b\in C^*_\rho(\G)$.\footnote{Actually, this characterizes bi-exactness -- see \cite[Lemma 15.1.4]{Me-n-Taka}.} 
\end{sublem}

\begin{proof} We define a C$^*$-algebra $D$ by  
\[
D=C^*(C^*_\lambda(\G),C^*_\rho(\G),\ell^\infty(\G))+\K(\G;\grp)\subset\B(\ell^2(\G)).
\]
Evidently $\K(\G;\grp)$ is an ideal in $D$ and 
$D/\K(\G;\grp)$ is a quotient of 
the crossed product of $\ell^\infty(\G)/c_0(\G;\grp)$ by $\G\times\G$ (actually, it's isomorphic to this crossed product). 
By assumption, the canonical $*$-homomorphism 
$C^*_\lambda(\G)\odot C^*_\rho(\G)\to D/\K(\G;\grp)$ 
is $\min$-continuous and $D/\K(\G;\grp)$ is nuclear (since we assumed $\G \times \G$ acts amenably on $\ell^\infty(\G)/c_0(\G;\grp)$ -- see Theorem \ref{thm:crossedproductamenableaction}). 
Hence, by the Choi-Effros Lifting Theorem (\cite[Theorem C.3]{Me-n-Taka}), the quotient map from $D$ to $D/\K(\G;\grp)$ has a u.c.p.\ splitting on any separable C$^*$-subalgebra. In particular, we can find a u.c.p.\ splitting $\theta \colon C^*_\lambda(\G)\otimes C^*_\rho(\G) \to D \subset \B(\ell^2(\G))$, which completes the proof. \end{proof}

\begin{subprop}\label{prop:injcric}
Let $M\subset\B(\hh)$ be a finite von Neumann algebra and $p\in M$ be a projection. 
Let $P\subset pMp$ be a von Neumann subalgebra and $E_P\colon pMp \to P$ be 
the trace-preserving 
conditional expectation. 
Consider the bi-normal u.c.p.\ map 
\[
\Phi_P\colon M\odot M'\ni\sum_k a_k\otimes b_k\mapsto
\sum_k E_P(pa_kp)b_kp\in\B(p\hh).
\]
Suppose that there are weakly dense $\mathrm{C}^*$-subalgebras 
$C_l\subset M$ and $C_r\subset M'$ 
such that $C_l$ is exact and $\Phi_P$ is $\min$-continuous on $C_l\odot C_r$. 
Then $P$ is injective.  
\end{subprop}
\begin{proof}
It can be shown that our assumptions imply that 
$\Phi_P$ is $\min$-continuous on $M\odot M'$ (cf.\ \cite[Lemma 9.2.9]{Me-n-Taka}). 
By The Trick, $\Phi_P|_M$ extends to a u.c.p.\ map 
$\psi$ from $\B(\hh)$ into $(pM')'=pMp$. 
(Note that the argument for The Trick only requires 
$\Phi_P|_{\C1\otimes M'}$ to be $*$-homomorphic.) 
Thus $E_P\circ\psi|_{\B(p\hh)}$ 
is a conditional expectation from $\B(p\hh)$ onto $P$.
\end{proof}

We primarily consider $\Phi_P$ in the case where $P=B'\cap pMp$ 
for a projection $p\in M$ and a \emph{diffuse}  
abelian von Neumann subalgebra $B\subset pMp$. 
Every diffuse abelian von Neumann algebra 
$B$ with separable predual is $*$-isomorphic to 
$L^\infty[0,1]$ and hence is generated by 
a single unitary element $u_0\in B$ (e.g., $u_0(t)=e^{2\pi i t}$). 
Fixing such a generator, we define a c.p.\ map $\Psi_B$ from $\B(\hh)$ 
into $\B(p\hh)$ by 
\[
\Psi_B(x)=\mbox{ultraweak-}\lim_n \frac{1}{n}\sum_{k=1}^n u_0^k x u_0^{-k},
\]
where the limit is taken along some fixed ultrafilter. 
It is not hard to see that $\Psi_B$ is a (nonunital) conditional 
expectation onto $B'\cap\B(p\hh)$ and that $\Psi_B|_{pMp}$ 
is a trace-preserving conditional expectation from $pMp$ onto $B'\cap pMp$. 
By uniqueness, one has 
$\Psi_B(a)=E_P(pap)$ for every $a\in M$. 
It follows that 
\[
\Psi_B(\sum_k a_kb_k)=\sum_k E_P(pa_kp)b_kp=\Phi_P(\sum_k a_k\otimes b_k)
\]
for $a_k\in M$ and $b_k\in M'$.

\subsection{A Bi-exact Version of Theorem \ref{thm:taka}}

Aren't the closing remarks of the preceding subsection reminiscent of the proof of Theorem \ref{thm:taka}?  Well, we'll soon use them to prove an analogue of that result in the context of bi-exactness.  But first, we have to recall an important theorem (of Popa) that is also needed in the proof (see \cite[Appendix F]{Me-n-Taka} for details). 

If $A\subset M$ are finite von Neumann algebras, then $\langle M,A\rangle$ denotes the algebra arising from Jones's basic construction.

\begin{subthm}\label{thm:popafinbimod}
Let $A\subset M$ be finite von Neumann algebras with separable predual
and let $p\in M$ be a nonzero projection.
Then, for a von Neumann subalgebra $B\subset pMp$,
the following are equivalent:
\begin{enumerate}
\item\label{thm:popafinbimod1}
there is no sequence $(w_n)$ of unitary elements\footnote{A unitary
element $w$ in $B$ is a partial isometry in $M$ such that $w^*w=p=ww^*$.}
in $B$ such that
$\|E_A(b^*w_na)\|_2\to0$ for every $a,b\in M$;
\item\label{thm:popafinbimod2}
there exists a positive element $d\in\langle M,A\rangle$ with
$\Tr(d)<\infty$ such that the ultraweakly closed convex hull of
$\{ w^*dw : w\in B\mbox{ unitary}\}$ does not contain $0$;
\item\label{thm:popafinbimod3}
there exists a $B$-$A$-submodule $\hh$ of $pL^2(M)$
with $\dim_A\hh<\infty$;
\item\label{thm:popafinbimod4}
there exist nonzero projections $e\in A$ and $f\in B$,
a unital normal $*$-homomorphism $\theta\colon fBf\to eAe$ and
a nonzero partial isometry $v\in M$ such that
\[
\forall x\in fBf, \quad xv=v\theta(x)
\]
and such that $v^*v\in\theta(fBf)'\cap eMe$ and $vv^*\in (fBf)'\cap fMf$.
\end{enumerate}
\end{subthm}

\begin{subdefn}\label{defn:embinside}  
Let $A\subset M$ and $B\subset pMp$ be finite von Neumann algebras.
We say \emph{$B$ embeds in $A$ inside $M$}
if one of the conditions
in Theorem \ref{thm:popafinbimod} holds.
\end{subdefn}

Note that if there is a nonzero projection $p_0\in B$
such that $p_0Bp_0$ embeds in $A$ inside $M$, then
$B$ embeds in $A$ inside $M$ (as condition (\ref{thm:popafinbimod4})
in Theorem \ref{thm:popafinbimod} evidently implies).

The proof of the following corollary can be found in \cite[Appendix F]{Me-n-Taka}. 

\begin{subcor}\label{cor:popafinbimod}
Let $M$ be a finite von Neumann algebra with separable predual
and $(A_n)$ be a sequence of von Neumann subalgebras.
Let $N\subset pMp$ be a von Neumann subalgebra
such that $N$ does not embed in $A_n$ inside $M$ for any $n$.
Then, there exists a diffuse abelian von Neumann subalgebra
$B\subset N$ such that $B$ does not embed in $A_n$ inside $M$ for any $n$.
\end{subcor}

Now we're ready for the analogue of Theorem \ref{thm:taka}. 

\begin{subthm}\label{thm:gvnacn}
Assume $\G$ is bi-exact relative to $\grp$ and  
let $p\in L(\G)$ be a projection and $N\subset pL(\G)p$ be a von Neumann subalgebra. 
If the relative commutant $N'\cap pL(\G)p$ is noninjective, 
then there exists $\Lambda\in\grp$ such that $N$ embeds in $L(\Lambda)$ inside $L(\G)$. 
\end{subthm}

\begin{proof}
By contradiction,  suppose that the conclusion of the theorem is not true. 
Then, by Corollary \ref{cor:popafinbimod}, there is 
a diffuse abelian von Neumann subalgebra $B\subset N$ 
such that $B$ does not embed in $L(\Lambda)$ inside $M=L(\G)$ for any $\Lambda$. To get our contradiction, it suffices to show $B' \cap pL(\G)p$ is injective (cf.\ the proof of Theorem \ref{thm:taka}).  

To do this, we'll use Proposition~\ref{prop:injcric} (and the remarks and notation which follow it) with $M = L(\G)$, $P = B' \cap pMp$, $C_l = C^*_\lambda(\G) \subset M$ and $C_r = C^*_\rho(\G) \subset M'$. Thus, it suffices to show $\Phi_P$ is $\min$-continuous on $C^*_\lambda(\G) \odot C^*_\rho(\G)$.  But to do this, it would be enough to know that $\Phi_P|_{C^*_\lambda(\G) \odot C^*_\rho(\G)} =\Psi_B\circ\theta|_{C^*_\lambda(\G) \odot C^*_\rho(\G)}$, where $\theta$ is the map given to us by Lemma~\ref{lem:clrmid}.  Since we already observed that 
\[
\Phi_P(\sum_k a_k\otimes b_k) = \Psi_B(\sum_k a_kb_k)
\]
for $a_k\in M$ and $b_k\in M'$, and we know $\theta(\sum_k a_kb_k) - \sum_k a_kb_k \in \K(\G;\grp)$, our task is further reduced to proving that $\K(\G;\grp)\subset\ker\Psi_B$ -- if we can do this, the proof is complete.

With $A=L(\Lambda) \subset M$, observe that $\chi_\Lambda\in\ell^\infty(\G)\subset\B(L^2(M))$ 
is the Jones projection $e_A$ onto $L^2(A) \subset L^2(M)$ and hence 
$\chi_{s\Lambda}=\lambda_s e_A\lambda_s^*\in\langle M,A\rangle_+$ 
with $\Tr(\chi_{s\Lambda})=1$. 
Therefore, it follows from the definition of $\Psi_B$ that $\Psi_B(\chi_{s\Lambda})$ is 
a positive element in $p\langle M,A\rangle p\cap B'$ such that 
$\Tr(\Psi_B(\chi_{s\Lambda}))\le1$. Since we're assuming $B$ does not embed in $A$ inside $M$,  the ultraweakly closed convex hull of $\{ w^*dw : w\in B\mbox{ unitary}\}$ contains $0$ for every element $d \in \langle M,A\rangle$ of finite trace; since $w^*\Psi_B(\chi_{s\Lambda})w = \Psi_B(\chi_{s\Lambda})$ for all unitaries $w\in B$, it follows that $\Psi_B(\chi_{s\Lambda})=0$. Finally, since $C^*_\rho(\G)$ is in the multiplicative domain of $\Psi_B$,  this implies that $\Psi_B(\chi_{s\Lambda t})=0$ for every $s,t\in\G$ and $\Lambda\in\grp$,   
or equivalently, $\K(\G;\grp)\subset\ker\Psi_B$. 
\end{proof}

To the von Neumann algebraist that doesn't yet appreciate C$^*$-algebras, I have a challenge: Find a proof of the previous theorem that doesn't depend on C$^*$-theory.\footnote{My earlier challenge to C$^*$-algebraists has the weight of history on its side -- unlike this W$^*$-challenge.  Hence, I'll be less surprised (but still quite surprised!) if someone finds a W$^*$-proof of Theorem \ref{thm:gvnacn} someday. Indeed, similar results have already been proved without appealing to C$^*$-algebras (cf.\ \cite{popa:productgap}, \cite{ipp}, \cite{peterson}).} The next two pages demonstrate this theorem's  remarkable power.

\subsection{More Solid Factors}

Here's a simple application of Theorem \ref{thm:gvnacn}. 

\begin{subcor}  
\label{cor:bisolid}
If $\G$ is bi-exact relative to the trivial subgroup $\{e\}$, then $L(\G)$ is solid.  
\end{subcor} 

\begin{proof} Let $B \subset L(\G)$ be diffuse and assume $B' \cap L(\G)$ is not injective. Then, by Theorem \ref{thm:gvnacn}, $B$ embeds in $L(\{e\})$ inside $L(\G)$. By condition (4) in Theorem \ref{thm:popafinbimod}, we can find a projection $f \in B$ and a unital $*$-homomorphism $fBf \to L(\{e\}) = \C$. Thus $\C$ is a direct summand of $fBf$, which implies $fBf$ has a minimal projection. This  contradicts our assumption that $B$ is diffuse. 
\end{proof} 

We've already seen that hyperbolic groups satisfy the hypothesis of the previous result, but there are others.  For example, \cite[Corollary 15.3.9]{Me-n-Taka} states that the wreath product of an amenable group by a group which is bi-exact relative to $\{e\}$ is again bi-exact relative to $\{e\}$.

\subsection{Product Groups} 

The proof of the following lemma is a good exercise. 

\begin{sublem}\label{lem:dpvnac}
Let $\G_1,\ldots,\G_n$ be groups that are bi-exact relative to $\{e\}$ and $\G=\prod_{i=1}^n\G_i$ be the direct product.
Define 
a family $\grp$ of subgroups of $\G$ by 
\[
\grp=\bigcup_{i=1}^n \{ \{e\} \times\prod_{j\neq i}\G_j \}.
\]
Then $\G$ is bi-exact relative to $\grp$.
\end{sublem}

Here's a lovely theorem of Ozawa and Popa.

\begin{subthm}
Let $\G_1,\ldots,\G_n$ be groups that are bi-exact relative to $\{e\}$ (e.g., hyperbolic groups) 
and $N_1,\ldots,N_m$ be noninjective $\mathrm{II}_1$-factors. 
If there exists an embedding 
\[
N_1\vt \cdots\vt N_m\hookrightarrow pL(\G_1\times\cdots\times\G_n)p,
\]
for some projection $p\in L(\G_1\times\cdots\times\G_n)$, then $m\le n$.
\end{subthm}
\begin{proof} The proof is by induction on $n$. If $n =1$, we're simply claiming that noninjective subfactors of $L(\G_1)$ are prime -- which follows from Corollary \ref{cor:bisolid}.  So assume the result holds for all products of length $n-1$ and we have an embedding $N_1\vt \cdots\vt N_m\subset pL(\G_1\times\cdots\times\G_n)p$. 

Since $$(N_1\vt \cdots\vt N_{m-1})' \cap pL(\G_1\times\cdots\times\G_n)p \supset N_m,$$ we see that  this relative commutant is noninjective.  Thus, Theorem~\ref{thm:gvnacn} and Lemma~\ref{lem:dpvnac} imply that, after permuting indices, one has 
\[
e_1N_1e_1\vt \cdots\vt N_{m-1} 
\hookrightarrow 
p_0L(\G_1\times\cdots\times\G_{n-1})p_0
\]
for some nonzero projections $e_1\in N_1$ and $p_0\in L(\G_1\times\cdots\times\G_{n-1})$.
By the induction hypothesis, $m-1 \leq n-1$, so we're done. 
\end{proof}  

In particular, $L(\F_2) \vt L(\F_2)$ and $N_1 \vt N_2 \vt N_3$ (for noninjective II$_1$-factors $N_i$) can't be isomorphic, which extends Ge's theorem that free group factors are prime.  However, free entropy degenerates in tensor products, so Ge's techniques can't be adapted to handle the case of higher tensor powers.

\subsection{Free Products} 

Inspired by the isomorphism problem for free group factors, we end this paper with one more application of Theorem \ref{thm:gvnacn}. But first we'll need two more of Popa's W$^*$-theorems.  Here's an older result on normalizers in free products (cf.\ \cite{popa:orthogonal}). 

\begin{subthm} 
\label{thm:freenormalizer} If $M_1$ and $M_2$ are finite von Neumann algebras and $N_1 \subset M_1$ is diffuse, then there is no unitary $u \in M_1 \ast M_2$ such that $u N_1 u^* \subset M_2$.  Moreover, if $u N_1 u^* \subset M_1$, then $u \in M_1$. 
\end{subthm} 

In particular, this result implies that $N_1^\prime \cap (M_1 \ast M_2) = N_1^\prime \cap M_1$. (It is crucial that $N_1$ be diffuse -- this statement is false if $N_1 \cong \M_n(\C)$!) 

The second thing we need is an improvement of Theorem \ref{thm:popafinbimod} in the case of factorial commutants. (See \cite[Lemma F.18]{Me-n-Taka} for a more general result.) 

\begin{subthm} 
\label{thm:finbimodII} Assume $M_1$ and $M_2$ are II$_1$-factors, $B \subset M_1 \ast M_2$ is diffuse and $B' \cap (M_1\ast M_2)$ is a factor.  If $B$ embeds in $M_1$ inside $M_1 \ast M_2$, then there exists a unitary $u \in M_1 \ast M_2$ such that $uB u^* \subset M_1$. 
\end{subthm} 

To apply Theorem \ref{thm:gvnacn} to free products, we need to know that $\G = \G_1 \ast \G_2 \ast \cdots \ast \G_n$ is bi-exact relative to the family of subgroups $\grp = \{ \G_1, \G_2, \ldots, \G_n\}$, whenever each $\G_i$ is exact.  This isn't trivial, but it is true -- see \cite[Proposition 15.3.12]{Me-n-Taka}. 

Combining these results, we have an important lemma. 

\begin{sublem} 
\label{lem:thelem2} Assume $\G = \G_1 \ast \G_2 \ast \cdots \ast \G_n$, with each $\G_i$ exact and i.c.c.,\footnote{Recall that this means $\{ gsg^{-1} : g \in \G_i\}$ is infinite for all $s \neq e$, and it ensures that $L(\G_i)$ is a factor.}  and $N = N_1 \vt N_2 \subset L(\G)$ has the following properties: Each $N_i$ is a II$_1$-factor, $N_1^\prime \cap L(\G) = N_2$ and $N_2$ is not injective.  Then, there exists $i \in \{1,\ldots,n\}$ and a unitary $u \in L(\G)$ such that $u N u^* \subset L(\G_i)$. 
\end{sublem} 

\begin{proof} By Theorem \ref{thm:gvnacn}, there is an $i$ such that $N_1$ embeds in $L(\G_i)$ inside $L(\G)$. Since $N_1$ is diffuse and $N_1^\prime \cap L(\G) = N_2$ is a factor, Theorem \ref{thm:finbimodII} provides a unitary $u$ such that $uN_1 u^* \subset L(\G_i)$. From Theorem \ref{thm:freenormalizer} we have that $(uN_1 u^*)' \cap L(\G) \subset L(\G_i)$. Since $uN_2 u^* = (uN_1 u^*)' \cap L(\G)$, the proof is complete. 
\end{proof} 

We say a group is a \emph{product group} if it's isomorphic to a Cartesian product $H \times K$ of groups. 

\begin{subthm}  Assume $\G_1,\ldots, \G_n$ and $\Lambda_1,\ldots,\Lambda_m$ are nonamenable, i.c.c., exact product groups.  If $$L(\G_1 \ast \cdots \ast \G_n) \cong L(\Lambda_1 \ast \cdots \ast \Lambda_m),$$ then $n = m$ and, modulo permutation of indices, $L(\G_i) \cong L(\Lambda_i)$ for all $i$. 
\end{subthm} 

\begin{proof} First write $\Lambda_i = H_i \times K_i$ where $K_i$ is nonamenable.  Then, by Theorem \ref{thm:freenormalizer}, $$L(H_i)' \cap L(\Lambda_1 \ast \cdots \ast \Lambda_m) = L(K_i)$$ is a noninjective II$_1$-factor. Thus, applying the previous lemma to each $L(\Lambda_i)$, we can find unitaries $u_i$ and indices $j(i)$ such that $u_i L(\Lambda_i) u_i^* \subset L(\G_{j(i)})$.  Similarly, we can find unitaries $v_j$ and indices $i(j)$ such that $v_j L(\G_j) v_j^* \subset L(\Lambda_{i(j)})$.  Hence, $$v_{j(i)} u_i L(\Lambda_i) u_i^* v_{j(i)}^* \subset L(\Lambda_{i(j(i))}).$$ 

Theroem \ref{thm:freenormalizer} now implies that $i(j(i)) = i$ and $v_{j(i)} u_i \in L(\Lambda_i) $.  Thus $v_{j(i)} u_i L(\Lambda_i) u_i^* v_{j(i)}^* = L(\Lambda_{i})$ and so $u_i L(\Lambda_i) u_i^* = L(\G_{j(i)})$. Finally, since $i(j(i)) = i$, it follows that $m \leq n$; by symmetry, we're done. 
\end{proof} 

This result can be improved  a bit (cf.\ \cite[Corollary 15.3.15]{Me-n-Taka}), but it's good enough to imply the following striking result: If $R$ denotes the injective II$_1$-factor, then $$\freeprod_1^n (R\vt L(\F_2)) \cong \freeprod_1^m (R\vt L(\F_2)) \Longleftrightarrow n = m.$$ Without the $R$'s, this is the isomorphism problem for free group factors (cf.\ \cite{radulescu}, \cite{dykema}).

\bibliographystyle{amsalpha}

\end{document}